\documentclass[12pt]{amsart}

\usepackage{amscd,amssymb,amsmath,amsthm,mathrsfs}
\usepackage{amsfonts}
\usepackage{fullpage}
\usepackage{enumitem} 
\usepackage{hyperref}
\usepackage{pdfsync}

\usepackage{color}
\usepackage[backrefs, alphabetic]{amsrefs}
\usepackage[all]{xy} 

\newtheorem{maintheorem}{Theorem}

\newtheorem{proposition}{Proposition}[section]
\newtheorem{theorem}[proposition]{Theorem}
\newtheorem{lemma}[proposition]{Lemma}
\newtheorem{fact}[proposition]{Fact}
\newtheorem*{fact*}{Fact}

\newtheorem*{claim*}{Claim}

\theoremstyle{definition}

\newtheorem{remark}[proposition]{Remark}

\numberwithin{equation}{section}

\DeclareMathOperator{\Char}{Char}

\DeclareMathOperator{\trdeg}{tr.deg}

\renewcommand{\H}{\operatorname{H}}

\DeclareMathOperator{\D}{D}
\DeclareMathOperator{\I}{I}
\DeclareMathOperator{\Hom}{Hom}

\newcommand{\Z}{\mathbb{Z}}
\newcommand{\Q}{\mathbb{Q}}

\newcommand{\Nbb}{\mathbb{N}}

\newcommand{\Vc}{\mathcal{V}}

\newcommand{\Oc}{\mathcal{O}}

\newcommand{\Cc}{\mathcal{C}}

\newcommand{\mf}{\mathfrak{m}}

\newcommand{\gfrak}{\mathfrak{g}}

\newcommand{\Mbf}{\mathbf{M}}
\newcommand{\Nbf}{\mathbf{N}}
\newcommand{\Rbf}{\mathbf{R}}
\renewcommand{\I}{\operatorname{I}}
\renewcommand{\D}{\operatorname{D}}

\newcommand{\Ivis}{\mathfrak{I}_{\rm vis}}

\newcommand{\Nbar}{\overline{\Nbb}}

\newcommand{\rl}{\operatorname{rank}_\ell}
\newcommand{\rr}{\operatorname{rank}_\Q}

\begin{document}
\title{Abelian-by-Central Galois Groups of Fields II:\\ Definability of Inertia/Decomposition Groups}
\author{Adam Topaz}
\thanks{This research was supported by NSF postdoctoral fellowship DMS-1304114.}
\address{Adam Topaz \vskip 0pt
Department of Mathematics \vskip 0pt
University of California, Berkeley \vskip 0pt
970 Evans Hall \#3840 \vskip 0pt
Berkeley, CA. 94720-3840 \vskip 0pt
USA}
\email{atopaz@math.berkeley.edu}
\urladdr{http://math.berkeley.edu/~atopaz}
\date{\today}

\begin{abstract}
	This paper explores some first-order properties of commuting-liftable pairs in pro-$\ell$ abelian-by-central Galois groups of fields.
	The main focus of the paper is to prove that minimized inertia and decomposition groups of many valuations are first-order definable using a predicate for the collection of commuting-liftable pairs.
\end{abstract}

\maketitle

\tableofcontents

\section{Introduction}

\emph{Birational anabelian geometry} is a subject where one tries to reconstruct fields of arithmetic and/or geometric significance from their Galois groups.
Most strategies in birational anabelian geometry have two main steps: the \emph{local theory} and the \emph{global theory}.
In the local theory, one tries to recover as much information as possible about the inertia and decomposition structure of valuations using the given Galois theoretical data.
And in the global theory, one tries to make sense of the local data to obtain meaningful information about the field in question.
This paper concerns the \emph{local theory} in birational anabelian geometry.

The history behind the local theory in birational anabelian geometry is quite rich, but we will focus on more recent developments in this overview.
For a more comprehensive discussion, see the introduction of \cite{Topaz2012c}.
On the one hand, one has the local theories which use \emph{large Galois groups} as their input.
For instance, one can recover the inertia and decomposition groups of \emph{$\ell$-tamely branching valuations} using the structure of the maximal pro-$\ell$ Galois group of a field which contains $\mu_\ell$; see \cite{Engler1994}, \cite{Efrat1995} and \cite{Engler1998}.
One can also reconstruct inertia and decomposition groups of certain valuations in absolute Galois groups of \emph{arbitrary fields} \cite{Koenigsmann2003}.

On the other hand, it has recently become apparent that much smaller Galois groups suffice to detect valuations.
Such local theories attempt to recover information about inertia/decomposition groups using \emph{abelian-by-central} Galois groups, or other similar ``almost-abelian'' invariants of a field, such as the Milnor K-ring or the Galois cohomology ring.
The first such theory was originally proposed by {\sc Bogomolov} \cite{Bogomolov1991} then further developed by {\sc Bogomolov-Tschinkel} \cite{Bogomolov2007} in the context of function fields over $\ell$-closed fields.
Furthermore, the mod-$2$ abelian-by-central context was first explored by \cite{Mah'e2004} with relation to valuations and orderings.
It was then shown that the mod-$\ell$ abelian-by-central Galois group encodes the existence of a tamely-branching $\ell$-Henselian valuation \cite{Efrat2011a}, based on results that detect valuations using mod-$\ell$ Milnor K-rings \cite{Efrat1999} \cite{Efrat2006b}.
Finally, \cite{Topaz2012c} shows that the \emph{minimized} inertia and decomposition groups of \emph{almost arbitrary} valuations can be recovered using the mod-$\ell^n$ abelian-by-central Galois group.
The one thing that most of these local theories have in common, the most general results of \cite{Topaz2012c} in particular, is that the recipe to recover inertia and decomposition groups is inherently \emph{second-order} and \emph{non-effective}, since the recipe involves looking for maximal subgroups which satisfy certain properties.

The present paper extends the local theories which recover inertia and decomposition groups using the mod-$\ell^n$ abelian-by-central Galois group.
The main property that sets this paper apart from its predecessors, is that the recipes described here are inherently \emph{first order}.
In more precise terms, in this paper we will show that the \emph{minimized} inertia and decomposition groups of many valuations are (uniformly) first-order definable by \emph{explicit} formulas in a natural language of abelian-by-central groups, given a suitable definable set of parameters (which exists in most situations).

A more broad goal of this paper is to initiate the development of an \emph{effective} version of Bogomolov's program in birational anabelian geometry.
This program, which was first introduced in \cite{Bogomolov1991}, aims to reconstruct higher-dimensional function fields over an algebraically closed field from their pro-$\ell$ abelian-by-central Galois groups.
The program was later formulated into a precise functorial conjecture by {\sc Pop} \cite{Pop2011}, and this conjecture is now commonly referred to as the \emph{Bogomolov-Pop conjecture} in birational anabelian geometry.
See \cite{Pop2012a}, \cite{Pop2011} or \cite{Topaz2013rational} for the precise formulation of the Bogomolov-Pop conjecture in birational anabelian geometry.

While the Bogomolov-Pop conjecture is still open in full generality, it has been proven in a few important cases by {\sc Bogomolov-Tschinkel} \cite{Bogomolov2008a} \cite{Bogomolov2011}, by {\sc Pop} \cite{Pop2003} \cite{Pop2011} \cite{Pop2011b}, and also by {\sc Silberstein} \cite{Silberstein2012}.
Nevertheless, the applicable local theory for the Bogomolov-Pop conjecture is by now well-developed; see \cite{Pop2010} and \cite{Pop2011a}.
In all known cases of the Bogomolov-Pop conjecture, the actual recipe which constructs the field from the given Galois group is, unfortunately, a second-order recipe, and one main reason for this is because of the local theory.
This paper therefore tackles the initial step in developing a first-order/effective version of Bogomolov's program, by providing a first-order recipe to determine the minimized inertia and decomposition groups of so-called \emph{quasi-divisorial} valuations.

\subsection{Minimized Galois Theory}

Throughout the note we will work with a fixed prime $\ell$.
We put $\Nbb := \{1,2,\ldots\}$ and $\Nbar := \Nbb \cup \{\infty\}$, with $\infty > n$ for all $n \in \Nbb$.
Throughout the note, we will also work with a fixed element $n \in \Nbar$.

For $m \in \Nbar$, we will consider fields $K$ such that $\mu_{\ell^m} \subset K$, but we impose no restrictions whatsoever on the characteristic of $K$.
Namely, the condition $\mu_{\ell^m} \subset K$ means simply that the polynomial $X^{\ell^m}-1$ splits completely in $K$ if $m \neq \infty$.
And the condition $\mu_{\ell^\infty} \subset K$ means that $X^{\ell^m}-1$ splits completely for all $m \in \Nbb$.
Because we impose no restrictions on the characteristic, we will need to work in the context of \emph{$\ell^m$-minimized Galois theory}, which we recall below.
The connection between the $\ell^m$-minimized theory and pro-$\ell$ Galois theory in the usual sense, for fields of characteristic $\neq \ell$, was the focus of the first paper in the series \cite{2013arXiv1310.5613T}; see Remark \ref{remark: Zassenhass} for more details about this connection.

For $m \in \Nbar$, we have a corresponding coefficient ring defined as
\[ \Lambda_m := \begin{cases}
	\Z/\ell^m, & m \neq \infty. \\
	\Z_\ell, & m = \infty.
\end{cases}\]
For a field $K$, we define the {\bf $\ell^m$-minimized Galois group of $K$} as follows:
\[ \gfrak^m(K) := \Hom(K^\times,\Lambda_m). \]
We will endow $\gfrak^m(K)$ with the \emph{point-wise convergence topology} which makes $\gfrak^m(K)$ into an abelian pro-$\ell$ group of exponent $\ell^m$.
For a subset $\Sigma \subset \gfrak^m(K)$, we recall that the {\bf orthogonal of $\Sigma$} is the subgroup of $K^\times$ defined as follows:
\[ \Sigma^\perp := \bigcap_{\sigma \in \Sigma} \ker\sigma. \]

Our main theorems will have an assumption on $K$ of the form $\mu_{2\ell^m} \subset K$.
If $\ell \neq 2$, we note that this is equivalent to the usual assumption $\mu_{\ell^m} \subset K$.
In general, we note that the assumption $\mu_{2\ell^m} \subset K$ ensures that $\sigma(-1) = 0$ for all $\sigma \in \gfrak^m(K)$.

A pair of elements $\sigma,\tau \in \gfrak^m(K)$ will be called a {\bf C-pair} provided that the following condition holds true: For all $x \in K\smallsetminus \{0,1\}$, one has 
\[ \sigma(x)\tau(1-x) = \sigma(1-x) \tau(x). \]
A subset $\Sigma$ of $\gfrak^m(K)$ will be called a {\bf C-set} if any pair of elements $\sigma,\tau \in \Sigma$ forms a C-pair.
Note that $\Sigma$ is a C-set if and only if $\langle \Sigma \rangle_{\Lambda_m}$ is a C-set, where $\langle \Sigma \rangle_{\Lambda_m}$ denotes the (closed) subgroup of $\gfrak^m(K)$ generated by $\Sigma$.

\begin{remark}[Connection with Galois Theory]
\label{remark: Zassenhass}
	For simplicity of notation, we assume in this remark that $m\neq\infty$; the case $m = \infty$ works in a similar way by passing to the limit.
	Suppose that $K$ is a field such that $\Char K \neq \ell$ and $\mu_{2\ell^m} \subset K$, and let $G_K$ denote the absolute Galois group of $K$.
	We recall that the first two non-trivial terms in the {\bf $\ell^m$-Zassenhauss filtration} of $G_K$ are defined as follows:
	\begin{enumerate}
		\item $G_K^{(2)} = [G_K,G_K] \cdot (G_K)^{\ell^m}$.
		\item $G_K^{(3)} = [G_K,G_K^{(2)}] \cdot (G_K)^{\delta \cdot \ell^m}$, where $\delta = 1$ if $\ell \neq 2$ and $\delta = 2$ if $\ell = 2$.
	\end{enumerate}

	Choose a primitive $\ell^m$-th root of unity $\omega \in \mu_{\ell^m} \subset K$.
	With this choice, Kummer theory yields an \emph{isomorphism} of pro-$\ell$ groups:
	\[ \sigma \mapsto \sigma^\omega : G_K/G_K^{(2)} \rightarrow \gfrak^m(K) \]
	which is defined by the condition that $\sigma^\omega(x) = i$ if and only if $\sigma(\sqrt[\ell^m]{x}) = \omega^i \cdot \sqrt[\ell^m]{x}$.
	Furthermore, for $x \in K^\times$, we let $(x)$ denote the image of $x$ under the \emph{Kummer map} $K^\times \rightarrow \H^1(K,\mu_{\ell^m})$.
	With this notation, \cite[Theorem 4]{2013arXiv1310.5613T} can be summarized as the following fact.
	\begin{fact}
	\label{fact: Kummer-MS}
		In the notation above, let $\sigma,\tau \in G_K/G_K^{(2)}$ be given.
		Then the following conditions are equivalent:
		\begin{enumerate}
			\item For all $x,y \in K^\times$ such that $(x) \cup (y) = 0$, one has $\sigma^\omega(x)\cdot \tau^\omega(y) = \sigma^\omega(y)\cdot \tau^\omega(x)$.
			\item $(\sigma^\omega,\tau^\omega)$ is a C-pair (as defined above).
			\item There exist representatives $\tilde\sigma,\tilde\tau \in G_K$ of $\sigma,\tau$ such that $\tilde\sigma^{-1}\tilde\tau^{-1}\tilde\sigma\tilde\tau \in G_K^{(3)}$.
		\end{enumerate}
	\end{fact}
	Using Fact \ref{fact: Kummer-MS}, we see that essentially all of the results in this paper can be easily translated to the usual Galois theoretical setting, for fields $K$ such that $\Char K \neq \ell$ and $\mu_{\ell^m} \subset K$ for $m$ sufficiently large.
	However, the so-called \emph{minimized} context defined above is more general since fields $K$ of characteristic $\ell$ are allowed.
\end{remark}

\subsection{Minimized Decomposition Theory}

Suppose now that $v$ is a valuation of $K$.
We let $\Oc_v$ denote the valuation ring with valuation ideal $\mf_v$.
Furthermore, we let $vK$ denote the value group of $K$, and we let $Kv$ denote the residue field of $v$.
We let $U_v := \Oc_v^\times$ denote the group of $v$-units, and we let $U_v^1 := (1+\mf_v)$ denote the group of principal $v$-units.

The {\bf minimized inertia} resp. {\bf decomposition} groups of $v$ are defined as follows:
\[ \I_v^m := \Hom(K^\times/U_v,\Lambda_m) \ \text{resp.} \ \D_v^m := \Hom(K^\times/U_v^1,\Lambda_m). \]
Note that $\I_v^m \subset \D_v^m \subset \gfrak^m(K)$, and that both $\I_v^m$ and $\D_v^m$ are closed subgroups of $\gfrak^m(K)$.
The minimized inertia and decomposition groups agree with the usual inertia and decomposition groups of $v$ in the case where $\Char Kv \neq \ell$, via the identification described in Remark \ref{remark: Zassenhass}; see \cite[Proposition 9.2]{Topaz2012c} for the details.

\subsection{$m$-Lifts}
\label{subsection: intro / lifts}

Assume now that $n,m \in \Nbar$ and that $n \leq m$.
For an element $a \in \Lambda_m$, we let $a_n$ denote the image of $a$ under the canonical map $\Lambda_m \rightarrow \Lambda_n$.
Similarly, for an element $\sigma \in \gfrak^m(K)$, we let $\sigma_n$ denote the element of $\gfrak^n(K)$ defined by
\[ \sigma_n(x) = \sigma(x)_n. \]

Let $\sigma \in \gfrak^n(K)$ be given.
We will say that $\sigma' \in \gfrak^m(K)$ is an {\bf $m$-lift} of $\sigma$ provided that $\sigma'_n = \sigma$.
To simplify the exposition, if $\sigma_1,\ldots,\sigma_r$ is a collection of elements of $\gfrak^n(K)$, then we will say that $\sigma_1',\ldots,\sigma_r'$ are {\bf $m$-lifts} of $\sigma_1,\ldots,\sigma_r$ provided that $\sigma_i'$ is an $m$-lift of $\sigma_i$ for all $i = 1,\ldots,r$.

\subsection{Visible Valuations}
\label{subsection: intro / visible valuations}

Although the valuations which we consider in this paper are fairly general, we still need to impose some restrictions.
We will say that a valuation $v$ of $K$ is an {\bf $m$-visible valuation} provided that the following three conditions hold true:
\begin{enumerate}
	\item[(V1)] $vK$ contains no non-trivial $\ell$-divisible convex subgroups.
	\item[(V2)] $\gfrak^m(Kv)$ is not a C-set.
	\item[(V3)] If $w$ is a valuation of $Kv$ such that $\D_w^m = \gfrak^m(Kv)$, then one has $\I_v^m = 1$.
\end{enumerate}
It turns out that most valuations which are of interest in anabelian geometry are indeed $m$-visible (for all $m$).
For instance, if $v$ is a valuation such that $vK$ contains no non-trivial $\ell$-divisible convex subgroups and such that $Kv$ is a function field of transcendence degree $\geq 1$ over an algebraically closed field, then $v$ is $m$-visible for all $m$.
In the notation of \S\ref{subsection: intro / quasi-divisorial valuations}, all quasi-divisorial valuations are visible; see Lemma \ref{lemma: quasi-divisors are visible}.

We will let $\Ivis^m(K)$ denote set of {\bf $m$-visible inertia elements}, defined as
\[ \Ivis^m(K) = \bigcup_{v \ \text{$m$-visible}} \I_v^m. \]
Our primary main theorem shows that this set of visible inertia elements is $\varnothing$-definable in a suitable language of C-pairs.

\subsection{The Cancellation Principle}
We will need to work with a few auxiliary elements of $\Nbar$ which depend on $n$ and $\ell$.
For $n,r \in \Nbb$, we define:
\begin{enumerate}
	\item $\Mbf_r(n) := (r+1) \cdot n - r$.
	\item $\Nbf(n) := \Mbf_1((6 \cdot \ell^{3n-2}-7) \cdot (n-1) + 3n-2)$.
	\item $\Rbf(n) := \Nbf(\Mbf_2(\Mbf_1(n)))$.
\end{enumerate}
We extend these definitions to $\Nbar$ by setting $\Mbf_r(\infty) = \Nbf(\infty) = \Rbf(\infty) = \infty$, to keep the notation consistent.
On the other hand, it is particularly important to note that 
\[ \Rbf(1) = \Nbf(1) = \Mbf_r(1) = 1.\]
Even though our main theorems deal with an arbitrary $n \in \Nbar$, this observation shows that the statements of our main theorems can be made significantly less technical if one restricts to the case where $n \in \{1,\infty\}$.
In general, one has the following important inequality:
\[ n \leq \Mbf_1(n) \leq \Mbf_2(\Mbf_1(n)) \leq \Rbf(n). \]

The precise formula for $\Nbf$ will not play any role in this paper.
This formula comes from the technical proof of the ``Main Theorem of C-pairs'' which appears in \cite[Theorem 3]{Topaz2012c}, and which is summarized in this paper as Theorem \ref{theorem: main theorem of C-pairs}.
It is important to note that we do not expect $\Nbf$ as above to be optimal.
Because of this, this paper has been written in such a way so that Theorem \ref{theorem: main theorem of C-pairs} is used solely as a black box, in order to account for possible future refinements of $\Nbf$.

On the other hand, the precise formula for $\Mbf_r(n)$ will be important because we will use the following ``cancellation principle'' extensively.
\begin{fact}[The Cancellation Principle]
\label{fact: cancellation}
	Let $r$ be a positive integer, and let $R \geq \Mbf_r(n)$ be given.
	Suppose that $c_1,\ldots,c_r \in \Lambda_R$ are given elements such that $(c_i)_n \neq 0$ for $i = 1,\ldots,r$, and suppose that $a,b \in \Lambda_R$ are such that $a \cdot c_1 \cdots c_r = b \cdot c_1 \cdots c_r$. 
	Then one has $a_n = b_n$.
\end{fact}

\subsection{The Language of C-pairs}
\label{subsection: language of C-pairs}
Since our goal will be to speak about definable sets of $\gfrak^n(K)$, we need to introduce the language which we consider.
First, for $m \in \Nbar$, we consider the structure $(\gfrak^m(K);\Cc^m)$ defined as follows:
\begin{enumerate}
	\item $\gfrak^m(K)$ is endowed with the usual structure of a group; i.e. the underlying language has a constant $0$ for the additive identity and a binary function $+$ for addition.
	\item $\Cc^m$ is a binary relation on $\gfrak^m(K)$, which is interpreted as: $(\sigma,\tau) \in \Cc^m$ if and only if $(\sigma,\tau)$ is a C-pair.
\end{enumerate}
For $n,m \in \Nbar$ such that $n \leq m$, we will also consider the two-sorted structure 
\[ (\gfrak^m(K),\gfrak^n(K); \Cc^m, \Cc^n,\pi) \]
defined as follows
\begin{enumerate}
	\item $(\gfrak^*(K);\Cc^*)$ is as defined above for $* = n,m$.
	\item $\pi : \gfrak^m(K) \rightarrow \gfrak^n(K)$ is the function $\sigma\mapsto \sigma_n$.
\end{enumerate}

\subsection{Main Theorems -- Defining Inertia}
\label{susbection: intro / defining inertia}
We are now prepared to state the main theorems of this paper which concern the definability of minimized inertia elements and minimized inertia groups of visible valuations.

\begin{maintheorem}
\label{maintheorem: definability of visible inertia set}
	Let $n \in \Nbar$ and $N \geq \Rbf(n)$ be given.
	Let $K$ be a field such that $\mu_{2\ell^N} \subset K$.
	For elements $\sigma \in \gfrak^n(K)$, the following are equivalent:
	\begin{enumerate}
		\item One has $\sigma \in \Ivis^n(K)$, i.e. there exists an $n$-visible valuation $v$ of $K$ such that $\sigma \in \I_v^n$.
		\item There exist $\tau_1,\tau_2 \in \gfrak^n(K)$ and $\sigma',\tau_1',\tau_2' \in \gfrak^N(K)$ such that the following conditions hold:
		\begin{enumerate}
			\item $(\tau_1,\tau_2)$ is not a C-pair.
			\item $\sigma',\tau_1',\tau_2'$ are $N$-lifts of $\sigma,\tau_1,\tau_2$.
			\item $(\sigma',\tau_1')$ and $(\sigma',\tau_2')$ are both C-pairs.
		\end{enumerate}
	\end{enumerate}
	In particular, the set $\Ivis^n(K)$ of $n$-visible inertia elements is $\varnothing$-definable in the two-sorted structure $(\gfrak^N(K),\gfrak^n(K);\Cc^N,\Cc^n,\pi)$.
\end{maintheorem}

\begin{maintheorem}
\label{maintheorem: definability of visible inertia groups}
	Let $n \in \Nbar$ and $N \geq \Rbf(n)$ be given.
	Let $K$ be a field such that $\mu_{2\ell^N} \subset K$.
	Let $\Sigma$ be any subset of $\gfrak^n(K)$.	
	Then the following conditions are equivalent:
	\begin{enumerate}
		\item There exists an $n$-visible valuation $v$ of $K$ such that $\Sigma \subset \I_v^n$.
		\item There exist $\tau_1,\tau_2 \in \gfrak^n(K)$ such that the following conditions hold true:
		\begin{enumerate}
			\item For all $\sigma,\tau \in \Sigma$, there exist $N$-lifts $\sigma',\tau'$ of $\sigma,\tau$ such that $(\sigma',\tau')$ form a C-pair.
			\item $(\tau_1,\tau_2)$ is not a C-pair.
			\item For all $\sigma \in \Sigma$, there exist $N$-lifts $\sigma',\tau_1',\tau_2'$ of $\sigma,\tau_1,\tau_2$ such that $(\sigma',\tau_1')$ and $(\sigma',\tau_2')$ are both C-pairs.
		\end{enumerate}
	\end{enumerate}
\end{maintheorem}
Theorem \ref{maintheorem: definability of visible inertia groups} will be primarily used in Theorem \ref{maintheorem: definability of visible decomposition groups} below as a technical condition for reconstructing the minimized inertia and decomposition groups of $n$-visible valuations.

\subsection{Main Theorem -- Defining Decomposition}
Our final main theorem will show how to reconstruct the minimized inertia and decomposition groups of $n$-visible valuations in an effective way.
We first need to introduce some technical notation which will be used in the statement of the theorem.
Let $m,n \in \Nbar$ be such that $n \leq m$.
For a subset $\Sigma$ of $\gfrak^n(K)$, we define two subsets $\D^m_n(\Sigma)$ and $\I^m_n(\Sigma)$ as follows.
\begin{enumerate}
	\item $\D^m_n(\Sigma)$ consists of all elements $\tau \in \gfrak^n(K)$ which satisfy the following condition: For all $\sigma \in \Sigma$, there exist $\tau_1,\tau_2 \in \gfrak^n(K)$, and $N$-lifts $\sigma',\tau',\tau_1',\tau_2'$ of $\sigma,\tau,\tau_1,\tau_2$ such that the following hold:
	\begin{enumerate}
		\item $(\tau_1,\tau_2)$ is not a C-pair.
		\item $(\sigma',\tau')$, $(\sigma',\tau_1')$ and $(\sigma',\tau_2')$ are all C-pairs.
	\end{enumerate}
	\item $\I^m_n(\Sigma)$ consists of all elements $\sigma \in \Sigma$ which satisfy the following condition: There exists an $m$-lift $\sigma'$ of $\sigma$ such that, for all $\tau \in \Sigma$, there exists an $m$-lift $\tau'$ of $\tau$, such that $(\sigma',\tau')$ is a C-pair.
\end{enumerate}

In the notation above, we will usually consider subsets $\Sigma$ of $\gfrak^n(K)$ as sets of \emph{parameters}, in order to construct the associated sets $\D^m_n(\Sigma)$ and $\I^m_n(\D^m_n(\Sigma))$.
In particular, if $\Sigma \subset \gfrak^n(K)$ is definable (with parameters, e.g. $\Sigma$ is finite) in the two-sorted structure $(\gfrak^m(K),\gfrak^n(K);\Cc^m,\Cc^n,\pi)$, then the two corresponding sets 
\[ \I^m_n(\D^m_n(\Sigma)) \subset \D^m_n(\Sigma) \]
are clearly definable as well.

\begin{remark}
\label{remark: centers centralizers}
	The precise definition of $\D^m_n$ and $\I^m_n$ is very technical primarily due to the fact that one needs to choose $m$-lifts of elements of $\gfrak^n(K)$.
	In the case where $n = m$, the situation becomes much simpler.
	Indeed, if $\Sigma$ satisfies the equivalent conditions of Theorem \ref{maintheorem: definability of visible inertia groups} (this will be an assumption in Theorem \ref{maintheorem: definability of visible decomposition groups}), then $\D^n_n(\Sigma)$ is precisely the set
	\[ \{\tau \in \gfrak^n(K) \ : \ \text{For all $\sigma \in \Sigma$, $(\sigma,\tau)$ is a C-pair}\}.\]
	Namely, $\D^n_n(\Sigma)$ is the ``C-centralizer'' of $\Sigma$.

	Similarly, for an arbitrary subset $\Sigma$ of $\gfrak^n(K)$, $\I^n_n(\Sigma)$ is precisely the set
	\[ \{\sigma \in \Sigma \ : \ \text{For all $\tau \in \Sigma$, $(\sigma,\tau)$ is a C-pair}\}.\]
	Namely, $\I^n_n(\Sigma)$ is the ``C-center'' of $\Sigma$.
\end{remark}

With this technical notation, we are finally prepared to state the main theorem concerning reconstructing the minimized inertia and decomposition groups of $n$-visible valuations.

\begin{maintheorem}
\label{maintheorem: definability of visible decomposition groups}
	Let $n \in \Nbar$ and $N \geq \Rbf(n)$ be given.
	Let $K$ be a field such that $\mu_{2\ell^N} \subset K$.
	Then the following hold:
	\begin{enumerate}
		\item Let $\Sigma$ be a subset of $\gfrak^n(K)$ which satisfies the equivalent conditions of Theorem \ref{maintheorem: definability of visible inertia groups}.
		Then there exists an $n$-visible valuation $v$ of $K$ such that 
		\[ \D^N_n(\Sigma) = \D_v^n \ \text{and} \ \I^N_n(\D^N_n(\Sigma)) = \I_v^n. \]
		Moreover, in this case one has $\Sigma \subset \I^N_n(\D^N_n(\Sigma))$.
		\item Conversely, if $v$ is an $n$-visible valuation and $\Sigma \subset \I_v^n$ is any subset such that $v(\Sigma^\perp)$ contains no non-trivial convex subgroups, then one has 
		\[ \D^N_n(\Sigma) = \D_v^n  \ \text{and} \ \I^N_n(\D^N_n(\Sigma)) = \I_v^n. \]
	\end{enumerate}
\end{maintheorem}

\begin{remark}
	This remark concerns the existence of $\Sigma$ as in Theorem \ref{maintheorem: definability of visible decomposition groups}(2).
	If $v$ is an $n$-visible valuation of $K$ and $\Sigma$ is a generating set of $\I_v^n$, then $v(\Sigma^\perp) = \ell \cdot vK$ contains no non-trivial convex subgroups, because $vK$ contains no non-trivial $\ell$-divisible convex subgroups.
	
	In fact, in most situations which are of interest in anabelian geometry, there exists a \emph{single element} $\sigma \in \I_v^n$ such that $v(\ker\sigma)$ contains no non-trivial convex subgroups.
	For instance, if $vK \cong \Gamma \times \Z$ ordered lexicographically for some totally ordered abelian group $\Gamma$, then one can take $\sigma$ to be the composition:
	\[ \sigma : K^\times \xrightarrow{v} vK \cong \Gamma \times \Z \twoheadrightarrow \Z \xrightarrow{\rm canonical} \Lambda_n. \]
	In this case, it follows from Theorem \ref{maintheorem: definability of visible decomposition groups}(2) that $\I_v^n$ and $\D_v^n$ are definable in the two-sorted structure $(\gfrak^N(K),\gfrak^n(K);\Cc^N,\Cc^n,\pi)$ with \emph{one parameter} from $\gfrak^n(K)$.
\end{remark}

\subsection{Quasi-Divisorial Valuations}
\label{subsection: intro / quasi-divisorial valuations}

Now we assume that $K$ is a function field over an algebraically closed field $k$.
We say that $v$ is a {\bf quasi-divisorial valuation of $K|k$} if $v$ is a valuation of $K$ such that the following hold:
\begin{enumerate}
	\item $vK$ contains no non-trivial $\ell$-divisible convex subgroups.
	\item One has $vK/vk \cong \Z$ as abstract groups.
	\item One has $\trdeg(K|k)-1 = \trdeg(Kv|kv)$.
\end{enumerate}
Quasi-divisorial valuations were first introduced by {\sc Pop} \cite{Pop2010} in the context of the local theory for the Bogomolov-Pop conjecture.
This terminology comes about from the fact that a quasi-divisovial valuation $v$ is \emph{divisorial}, i.e. it arises from a Weil-prime-divisor on some normal model of $K|k$, if and only if $vk = 0$.
As noted above, it turns out that quasi-divisorial valuations are $n$-visible for all $n \in \Nbar$; see Lemma \ref{lemma: quasi-divisors are visible} for the details.

We will conclude the paper by adapting the methods from \cite{Pop2010} and \cite{Pop2011a} in two ways: first, to work with a general $n \in \Nbar$ and second, to work with the more general ``definable'' framework introduced above.
This is summarized as the following theorem.

\begin{maintheorem}
\label{maintheorem: quasi-prime-divisors}
	Let $n \in \Nbar$ and $N \geq \Rbf(n)$ be given.
	Let $K$ be a function field over an algebraically closed field $k$ such that $d := \trdeg(K|k) \geq 2$.
	Let $I \subset D \subset \gfrak^n(K)$ be two subsets.
	Then the following are equivalent:
	\begin{enumerate}
		\item There exist $\sigma_1,\ldots,\sigma_{d-1}\in \gfrak^n(K)$ such that the following hold:
		\begin{enumerate}
			\item $\{\sigma_1,\ldots,\sigma_{d-1}\}$ satisfies the equivalent conditions of Theorem \ref{maintheorem: definability of visible inertia groups}.
			\item $\langle \sigma_1,\ldots,\sigma_{d-1} \rangle_{\Lambda_n}$ has rank $d-1$.
			\item $\D^N_n(\sigma_1) = D$ and $\Lambda_n \cdot \sigma_1 = \I^N_n(D) = I$.
		\end{enumerate}
		\item There exists a quasi-divisorial valuation $v$ of $K|k$ such that $I = \I_v^n$ and $D = \D_v^n$.
	\end{enumerate}
\end{maintheorem}

\begin{remark}
\label{remark: quasi-divisorial}
	Let $K$ be a function field of transcendence degree $\geq 2$ over an algebraically closed field $k$, and let $n \in \Nbar$ and $N \geq \Rbf(n)$ be given.
	Consider the set $\mathfrak{I}_{\rm q.d.}^n(K|k)$ of generators of minimized inertia groups of quasi-divisorial valuations of $K|k$:
	\[ \mathfrak{I}_{\rm q.d.}^n(K|k) := \{ \sigma \in \gfrak^n(K) \ : \ \text{$\Lambda_n \cdot \sigma = \I_v^n$ for some quasi-divisorial valuation $v$ of $K|k$}\}. \]

	In the case where $n \neq \infty$, it follows immediately from Theorems \ref{maintheorem: definability of visible inertia set} and \ref{maintheorem: quasi-prime-divisors} that the set $\mathfrak{I}_{\rm q.d.}^n(K|k)$ is $\varnothing$-definable in the two-sorted structure $(\gfrak^N(K),\gfrak^n(K);\Cc^N,\Cc^n,\pi)$.
	Moreover, Theorem \ref{maintheorem: quasi-prime-divisors} implies that the $\ell^n$-minimized inertia and decomposition groups of quasi-divisorial valuations of $K|k$ are uniformly definable with one parameter in $\mathfrak{I}_{\rm q.d.}^n(K|k)$.

	A similar definability result also holds for $n = \infty$, after enlarging the language to encode finitely-generated $\Lambda_\infty$-submodules of $\gfrak^\infty(K)$.
	To be precise, we consider the structure $(\gfrak^\infty(K);\Cc^\infty,\Delta_r)_{r \in \Nbb}$ where $\Delta_r$ is an $(r+1)$-ary relation interpreted as 
	\[ \Delta_r(\sigma_1,\ldots,\sigma_r;\tau) \Longleftrightarrow \tau \in \langle\sigma_1,\ldots,\sigma_r \rangle_{\Lambda_\infty}.\]
	Then the set $\mathfrak{I}_{\rm q.d.}^\infty(K|k)$ is $\varnothing$-definable in this enriched structure, and the $\ell^\infty$-minimized inertia and decomposition groups of quasi-divisorial valuations of $K|k$ are uniformly definable in this structure with one parameter in $\mathfrak{I}_{\rm q.d.}^\infty(K|k)$.

	The (uniform) formulas used to define the minimized inertia and decomposition groups of quasi-divisorial valuations clearly depend on $d = \trdeg(K|k)$.
	In Theorem \ref{theorem: transcendence degree}, we give a simple recipe to recover $d = \trdeg(K|k)$ using the structure $(\gfrak^m(K);\Cc^m)$ if $m \neq \infty$, or using the enriched structure $(\gfrak^\infty(K);\Cc^\infty,\Delta_r)_{r \in \Nbb}$.
	More precisely, it follows immediately from Theorem \ref{theorem: transcendence degree} that $\trdeg(K|k)$ is an invariant of the first-order theory of $(\gfrak^m(K);\Cc^m)$ for $m \neq \infty$ resp. $(\gfrak^\infty(K);\Cc^\infty,\Delta_r)_{r \in \Nbb}$.
\end{remark}

\section{Minimized Decomposition Theory and C-pairs}
\label{section: Minimized Decomposition Theory}

In this section, we will recall the required facts concerning the connection between C-pairs and minimized decomposition theory.
Most of the lemmas in this section can be found, at least in some form, in the more comprehensive paper \cite{Topaz2012c}.
However, in order to keep the discussion as self contained as possible, we will provide some of the less technical proofs here, while referring to loc.cit. for some technical results.
Throughout this section $K$ will be an arbitrary field, unless otherwise specified.

\subsection{C-pair Structure of Minimized Decomposition Groups}
\label{subsection: mindec / C-pair structure}

Let $m \in \Nbar$ be given.
Suppose that $v$ is a valuation of $K$ and let $\sigma \in \D_v^m$ be a given element.
Recall that $\sigma$ is by definition a homomorphism
\[ \sigma : K^\times/U_v^1 \rightarrow \Lambda_m .\]
We let $\sigma_v$ denote the restriction of $\sigma$ to $Kv^\times = U_v/U_v^1$.
Thus, the map $\sigma \mapsto \sigma_v$ yields a canonical homomorphism
\[ \D_v^m \rightarrow \gfrak^m(Kv). \]
Our first lemma proves some compatibility properties of this canonical map.
\begin{lemma}
\label{lemma: residue C-pair}
	Let $m \in \Nbar$ be given, and let $(K,v)$ be a valued field.
	Furthermore, let $w$ be a valuation of $Kv$.
	Then the following hold:
	\begin{enumerate}
		\item The canonical map $\D_v^m \rightarrow \gfrak^m(Kv)$ induces a canonical isomorphism
		\[ \D_v^m/\I_v^m \xrightarrow{\cong} \gfrak^m(Kv). \]
		\item Identifying $\D_v^m/\I_v^m$ with $\gfrak^m(Kv)$ as in (1) above, one has 
		\[ \D_w^m = \D_{w\circ v}^m/\I_v^m, \ \ \I_w^m = \I_{w \circ v}^m/\I_v^m. \]
		\item Let $\sigma,\tau \in \D_v^m$ be given such that $\sigma(-1) = \tau(-1) = 0$.
		Then $(\sigma,\tau)$ is a C-pair in $\gfrak^m(K)$ if and only if $(\sigma_v,\tau_v)$ is a C-pair in $\gfrak^m(Kv)$.
	\end{enumerate}
\end{lemma}
\begin{proof}
	We will assume that $m \neq \infty$, since the $m = \infty$ case would follow from the $m \neq \infty$ case by taking limits.

	\vskip 10pt
	\noindent\emph{Proof of (1).}
	Consider the following canonical short exact sequence:
	\[ 1 \rightarrow Kv^\times \rightarrow K^\times/U_v^1 \rightarrow vK \rightarrow 1.\]
	Since $vK$ is torsion-free, we obtain an induced short exact sequence by tensoring with $\Lambda_m$:
	\[ 1 \rightarrow Kv^\times/\ell^m \rightarrow K^\times/(K^{\times\ell^m} \cdot U_v^1) \rightarrow vK/\ell^m \rightarrow 1. \]
	Assertion (1) follows from Pontryagin duality by applying the functor $\Hom(\bullet,\Lambda_m)$ to this short exact sequence.

	\vskip 10pt
	\noindent\emph{Proof of (2).}
	The proof of assertion (2) follows in essentially the same way as  the proof of assertion (1), by considering the following two short exact sequences:
	\[ 1 \rightarrow Kv^\times/U_w^1 \rightarrow K^\times/U_{w \circ v}^1 \rightarrow vK \rightarrow 1 \]
	and 
	\[ 1 \rightarrow wK \rightarrow (w \circ v)K \rightarrow vK \rightarrow 1. \]

	\vskip 10pt
	\noindent\emph{Proof of (3).}
	If $(\sigma,\tau)$ is a C-pair, then clearly $(\sigma_v,\tau_v)$ is a C-pair as well.
	Conversely, assume that $(\sigma_v,\tau_v)$ is a C-pair, and let $x \in K \smallsetminus \{0,1\}$ be given.
	We will consider several cases, based on the values of $v(x)$ and $v(1-x)$.

	\vskip 10pt
	\noindent\underline{\bf Case $v(x) > 1$}:
	In this case, one has $1-x \in U_v^1 \subset \ker\sigma \cap \ker\tau$.
	Therefore, one has 
	\[ \sigma(x)\tau(1-x) = 0 = \sigma(1-x)\tau(x). \]

	\vskip 10pt
	\noindent\underline{\bf Case $v(x) < 1$}:
	In this case, one has $1-x \in (-x) \cdot U_v^1$ so that $\sigma(1-x) = \sigma(-x) = \sigma(x)$ and $\tau(1-x) = \tau(-x)=\tau(x)$.
	Therefore, one has 
	\[ \sigma(x)\tau(1-x) = \sigma(x)\tau(x) = \sigma(1-x)\tau(x). \]
	
	\vskip 10pt
	\noindent\underline{\bf Case $v(x) = 0$ and $v(1-x)>0$}:
	In this case, one has $x \in U_v^1 \subset \ker\sigma \cap \ker\tau$.
	Therefore, one has 
	\[ \sigma(x) \tau(1-x) = 0 = \sigma(1-x)\tau(x). \]

	\vskip 10pt
	\noindent\underline{\bf Case $v(x) = v(1-x) = 0$}:
	Let $z \mapsto \bar z$ denote the canonical map $U_v \rightarrow Kv^\times$.
	Since $(\sigma_v,\tau_v)$ is a C-pair, one has 
	\[ \sigma(x)\tau(1-x) = \sigma_v(\bar x)\tau_v(\bar 1-\bar x) = \sigma_v(\bar 1-\bar x)\tau_v(\bar x) = \sigma(1-x)\tau(x). \]
	
	In any case, we see that for all $x \in K \smallsetminus \{0,1\}$, one has 
	\[ \sigma(x)\tau(1-x) = \sigma(1-x)\tau(x) \]
	thus $(\sigma,\tau)$ is a C-pair, as required. 
\end{proof}

\subsection{Existence of Lifts}
\label{susbection: mindec / existence of lifts}

Suppose now that $m,n \in \Nbar$ are such that $n \leq m$.
It is easy to see that the canonical map $\gfrak^m(K) \rightarrow \gfrak^n(K)$ restricts to compatible maps
\[ \D_v^m \rightarrow \D_v^n, \ \ \I_v^m \rightarrow \I_v^n. \]
Our next lemma shows that these maps are all surjective in a fairly strong sense.
\begin{lemma}
\label{lemma: m-lifts}
	Let $m,n \in \Nbar$ be given such that $n \leq m$.
	Let $K$ be a field such that $\mu_{\ell^m} \subset K$, and let $v$ be a valuation of $K$.
	Then the following hold:
	\begin{enumerate}
		\item The canonical map $\gfrak^m(K) \rightarrow \gfrak^n(K)$ is surjective.
		\item The two pro-$\ell$ groups $\gfrak^m(K)$ and $\gfrak^n(K)$ have the same rank (as pro-$\ell$ groups).
		\item The canonical maps $\D_v^m \rightarrow \D_v^n$ and $\I_v^m \rightarrow \I_v^n$ are surjective.
		\item One has $\I_v^m = 1$ if and only if $\I_v^n = 1$.
		\item One has $\D_v^m = \gfrak^m(K)$ if and only if $\D_v^n = \gfrak^n(K)$.
	\end{enumerate}
\end{lemma}
\begin{proof}
	As in the proof of Lemma \ref{lemma: residue C-pair}, we will assume that $m \neq \infty$ and thus $n \neq \infty$, since the case where either $m$ or $n$ is $\infty$ would follow by passing to the limit.

	\vskip 10pt
	\noindent\emph{Proof of (1).}
	The Pyntryagin dual of the given map $\gfrak^m(K) \rightarrow \gfrak^n(K)$ is precisely 
	\[ K^\times/\ell^n \xrightarrow{\ell^{m-n}} K^\times/\ell^m. \]
	It is straightforward to verify that this map $K^\times/\ell^n \rightarrow K^\times/\ell^m$ is \emph{injective} since $K$ contains $\mu_{\ell^m}$.
	By Pontryagin duality, we deduce that the dual map $\gfrak^m(K) \rightarrow \gfrak^n(K)$ is surjective.

	\vskip 10pt
	\noindent\emph{Proof of (2).}
	Arguing similarly as in (1) above, we see that the kernel of the (surjective) map $\gfrak^m(K) \rightarrow \gfrak^n(K)$ is precisely $\ell^n \cdot \gfrak^m(K)$.
	Thus, the projection $\gfrak^m(K) \rightarrow \gfrak^n(K)$ yields an isomorphism of pro-$\ell$ abelian groups
	\[ \gfrak^m(K)/\ell^n \cong \gfrak^n(K). \]
	Therefore $\gfrak^m(K)$ and $\gfrak^n(K)$ have the same rank as pro-$\ell$ groups.

	\vskip 10pt
	\noindent\emph{Proof of (3).}
	It easily follows from the fact that $vK = K^\times/U_v$ is torsion-free that the map $\I_v^m \rightarrow \I_v^n$ is surjective.
	On the other hand, by (1) we also know that the map $\gfrak^m(Kv) \rightarrow \gfrak^n(Kv)$ is surjective since $\mu_{\ell^m} \subset Kv$.
	By Lemma \ref{lemma: residue C-pair}(1), one has a commutative diagram with exact rows:
	\[ 
	\xymatrix{
	1 \ar[r] & \I_v^m \ar[d]_{\text{surjective}} \ar[r] & \D_v^m \ar[d] \ar[r] & \gfrak^m(Kv) \ar[d]^{\text{surjective}} \ar[r] & 1 \\
	1 \ar[r] & \I_v^n \ar[r] & \D_v^n \ar[r] & \gfrak^n(Kv) \ar[r] & 1
	}
	\]
	The surjectivity of the map $\D_v^m \rightarrow \D_v^n$ follows by chasing the diagram above.

	\vskip 10pt
	\noindent\emph{Proof of (4).}
	If $\I_v^m = 1$ then clearly $\I_v^n = 1$ by (3).
	Conversely, if $\I_v^n = 1$ then $vK$ is $\ell^n$-divisible.
	But, as $vK$ is torsion-free, it follows that $vK$ is $\ell$-divisible, hence $\ell^m$-divisible as well.
	Therefore $\I_v^m = 1$.

	\vskip 10pt
	\noindent\emph{Proof of (5).}
	If $\D_v^m = \gfrak^m(K)$ then $\D_v^n = \gfrak^n(K)$ by (3).
	Conversely, assume that $\D_v^n = \gfrak^n(K)$.
	Then one has $U_v^1 \subset K^{\times\ell^n}$.

	Let $x \in \mf_v$ be given and consider $1-x \in U_v^1$.
	By the above, we see that there exists some $y \in K^\times$ such that $1-x = y^{\ell^n}$.
	Since $v(1-x) = 0$, it follows that $v(y) = 0$ as well.
	Let $z \mapsto \bar z$ denote the canonical map $U_v \rightarrow Kv^\times$.
	Then one has $\bar 1 = \bar y^{\ell^n}$.
	Since $\mu_{\ell^m} \subset Kv$, there exists some $z \in U_v$ such that $\bar y = \bar z^{\ell^{m-n}}$.
	Therefore, there exists some $w \in U_v^1$ such that $y = w z^{\ell^{m-n}}$ and thus
	\[ 1-x = w^{\ell^n} \cdot z^{\ell^m}.\]
	But $w \in U_v^1 \subset K^{\times\ell^n}$ and therefore $1-x \in K^{\times\ell^{2n}} \cdot K^{\times \ell^m}$.
	As $x \in \mf_v$ was arbitrary, we deduce that $U_v^1 \subset K^{\times \ell^{2n}} \cdot K^{\times \ell^m}$.
	Proceeding inductively in this way, we deduce that $U_v^1 \subset K^{\times \ell^m}$, hence $\D_v^m = \gfrak^m(K)$.
\end{proof}

\subsection{The Main Theorem of C-pairs}

Let $m \in \Nbar$ be given, and suppose that $v$ is a valuation of $K$.
Suppose that $\sigma,\tau \in \D_v^m$ are given such that $\langle \sigma,\tau \rangle_{\Lambda_m}/(\langle \sigma,\tau \rangle_{\Lambda_m} \cap \I_v^m)$ is $\Lambda_m$-cyclic.
Then by Lemma \ref{lemma: residue C-pair}, it follows that $(\sigma,\tau)$ is a C-pair.
The non-trivial direction in the ``Main Theorem of C-pairs'' can be seen as a weak converse to this fact.

One should note that the proof of the main theorem of C-pairs uses the classical theory of ``rigid elements'' in a fundamental way.
The theory of rigid elements shows how to construct a valuation ring given certain bounds on the units and principal units.
The theory was originally developed by {\sc Ware} \cite{Ware1981} and {\sc Arason-Elman-Jacob} \cite{Arason1987}, as well as others.
See also the summary of the main results from \cite{Arason1987} which appears in \cite[Theorem 4]{Topaz2012c}.

We will only prove the trivial direction of the main theorem of C-pairs, whereas the non-trivial direction (which uses rigid elements) can be found in \cite[Theorem 3]{Topaz2012c}, the proof of which appears in \S11 of loc.cit.

\begin{theorem}[Main Theorem of C-pairs]
\label{theorem: main theorem of C-pairs}
	Let $m \in \Nbar$ and $M \geq \Nbf(m)$ be given.
	Let $K$ be a field such that $\mu_{2\ell^M} \subset K$, and let $\sigma,\tau \in \gfrak^m(K)$ be given.	
	Then the following are equivalent:
	\begin{enumerate}
		\item There exists a valuation $v$ of $K$ such that $\sigma,\tau \in \D_v^m$ and $\langle \sigma,\tau \rangle_{\Lambda_m}/(\langle\sigma,\tau\rangle_{\Lambda_m} \cap \I_v^m)$ is $\Lambda_m$-cyclic.
		\item There exist $M$-lifts $\sigma',\tau'$ of $\sigma,\tau$ such that $(\sigma',\tau')$ is a C-pair in $\gfrak^M(K)$.
	\end{enumerate}	
\end{theorem}
\begin{proof}
	We give a full proof of the trivial direction, $(1) \Rightarrow (2)$, while referring to \cite{Topaz2012c} for the non-trivial direction $(2) \Rightarrow (1)$.

	\vskip 10pt
	\noindent$(1) \Rightarrow (2)$.
	By the assumption, there exist $f,g \in \langle \sigma,\tau \rangle_{\Lambda_m}$ such that $f \in \I_v^m$, $g \in \D_v^m$, and such that $\langle f,g \rangle_{\Lambda_m} = \langle \sigma,\tau \rangle_{\Lambda_m}$.
	By Lemma \ref{lemma: m-lifts}(3), there exist $M$-lifts $f',g'$ of $f,g$ such that $f' \in \I_v^M$ and $g' \in \D_v^M$.
	Finally, by Lemma \ref{lemma: residue C-pair}(3), $(f',g')$ is a C-pair, so $\langle f',g' \rangle_{\Lambda_M}$ is a C-set.
	Assertion (2) follows from the fact that $(\langle f',g' \rangle_{\Lambda_M})_m = \langle \sigma,\tau \rangle_{\Lambda_m}$.

	\vskip 10pt
	\noindent$(2) \Rightarrow (1)$.
	See \cite[Theorem 3]{Topaz2012c}.
\end{proof}

\section{Valuative Subsets}
\label{section: valuative subsets}

Let $m \in \Nbar$ be given, and let $K$ be an arbitrary field.
We say that a subset $\Sigma$ of $\gfrak^m(K)$ is a {\bf valuative subset} provided that there exists \emph{some} valuation $v$ of $K$ such that $\Sigma \subset \I_v^m$.
In other words, $\Sigma$ is valuative if and only if there exists a valuation $v$ of $K$ such that $U_v \subset \Sigma^\perp$.

\begin{lemma}
\label{lemma: valuative}
	Let $m \in \Nbar$ be given, and let $K$ be a field.
	Suppose that $\Sigma \subset \gfrak^m(K)$ is a valuative subset.
	Then there exists a unique coarsest valuation $v_\Sigma$ of $K$ such that $\Sigma \subset \I_{v_\Sigma}^m$.
	More precisely, if $w$ is any valuation of $K$ such that $\Sigma \subset \I_w^m$, then $v_\Sigma$ is the coarsening of $w$ associated to the maximal convex subgroup of $v(\Sigma^\perp)$.
\end{lemma}
\begin{proof}
	Let $w$ be any valuation such that $\Sigma \subset \I_w^n$.
	Equivalently, one has $U_w \subset \Sigma^\perp$.
	Let $v$ be the coarsening of $w$ associated to the maximal convex subgroup of $w(\Sigma^\perp)$.
	Thus, by construction, $v(\Sigma^\perp)$ contains no non-trivial convex subgroups and $U_v \subset \Sigma^\perp$.
	
	Consider the set
	\[ H := \{t \in \Sigma^\perp \ : \ \forall x \in K^\times \smallsetminus \Sigma^\perp, \ t-x \in (1-x) \cdot \Sigma^\perp\}.\]
	Since $U_v \subset \Sigma^\perp$, the ultrametric inequality immediately implies that $U_v \subset H$.
	We claim that $U_v = H$.

	Suppose that $t \in H$ is given and assume that $v(t) > 0$.
	Since $v(\Sigma^\perp)$ contains no non-trivial convex subgroups, there exists some $x \in K^\times \smallsetminus \Sigma^\perp$ such that $0 < v(x) < v(t)$.
	But then $t-x \in x \cdot U_v \subset x \cdot \Sigma^\perp$, while $1-x \in U_v \subset \Sigma^\perp$.
	This contradicts the definition of $H$.

	Similarly, if $v(t) < 0$, then there exists some $x \in K^\times \smallsetminus \Sigma^\perp$ such that $v(t) < v(x) < 0$.
	Thus $t-x \in t \cdot U_v \subset t \cdot \Sigma^\perp = \Sigma^\perp$, while $(1-x) \in x \cdot U_v \subset x \cdot \Sigma^\perp$.
	Again, this contradicts the definition of $H$.

	Therefore, we deduce that $H = U_v$.
	In particular, $U_v = H$ depends only on $\Sigma$ and $K$, but not at all on the original choice of valuation $w$.
	This completes the proof of the lemma.
\end{proof}

Given a valuative subset $\Sigma$ of $\gfrak^m(K)$, we will denote {\bf the valuation associated to $\Sigma$} by $v_\Sigma$ as discussed in Lemma \ref{lemma: valuative}.
Namely, $v := v_\Sigma$ is the unique valuation such that one has 
\[ U_v = \left\{t \in \Sigma^\perp \ : \ \forall x \in K^\times \smallsetminus \Sigma^\perp, \ t-x \in (1-x) \cdot \Sigma^\perp\right\}.\]
Note that if $v = v_\Sigma$ for some valuative subset $\Sigma$ of $\gfrak^m(K)$, then $vK$ contains no non-trivial $\ell$-divisible convex subgroups.
Indeed, any $\ell$-divisible convex subgroup must be contained in $v(\Sigma^\perp)$, and must therefore be trivial by Lemma \ref{lemma: valuative}.

\subsection{Comparability of Valuations}

In this subsection, we prove some lemmas concerning comparability of valuations associated to valuative subsets.
\begin{lemma}
\label{lemma: C-pair implies comparable}
	Let $m \in \Nbar$ and $M \geq \Mbf_1(m)$ be given.
	Let $K$ be a field such that $\mu_{2\ell^M} \subset K$, and let $\sigma,\tau \in \gfrak^m(K)$ be two valuative elements.
	Then the following are equivalent:
	\begin{enumerate}
		\item The two valuations $v_\sigma,v_\tau$ are comparable.
		\item There exist $M$-lifts $\sigma',\tau'$ of $\sigma,\tau$ such that $(\sigma',\tau')$ is a C-pair.
	\end{enumerate}
\end{lemma}
\begin{proof}
	The proof of this lemma relies on the theory of rigid elements.
	We will prove the trivial direction, as well as some of the non-technical details for the non-trivial direction, but we will refer to \cite{Topaz2012c} for the portion which uses rigid elements.

	\vskip 10pt
	\noindent$(1) \Rightarrow (2)$.
	Say, e.g. that $v_\sigma$ is coarser than $v_\tau$, so that $\I_{v_\sigma}^m \subset \I_{v_\tau}^m$, hence $\sigma,\tau \in \I_{v_\tau}^m$.
	By Lemma \ref{lemma: m-lifts}(3), there exist $M$-lifts $\sigma',\tau'$ of $\sigma,\tau$ such that $\sigma',\tau' \in \I_{v_\tau}^M$.
	By Lemma \ref{lemma: residue C-pair}(3), $(\sigma',\tau')$ is a C-pair, as required.

	\vskip 10pt
	\noindent$(2) \Rightarrow (1)$.
	It follows from Fact \ref{fact: cancellation} that for all $x \in K \smallsetminus \{0,1\}$, the subgroup
	\[ \langle (\sigma(x),\tau(x)),(\sigma(1-x),\tau(1-x)) \rangle_{\Lambda_m} \]
	of $\Lambda_m \times \Lambda_m$ is $\Lambda_m$-cyclic.
	Thus, (1) follows from \cite[Proposition 3.6]{Topaz2012c}.
\end{proof}

\begin{lemma}
\label{lemma: supremum}
	Let $m \in \Nbar$ be given, and let $K$ be a field.
	Let $\Sigma$ be a subset of $\gfrak^m(K)$ consisting of valuative elements such that, for all $\sigma,\tau \in \Sigma$, the two valuations $v_\sigma,v_\tau$ are comparable.
	Then $\Sigma$ is valuative, and $v_\Sigma$ is the valuation-theoretic supremum of $(v_\sigma)_{\sigma \in \Sigma}$.
	Moreover, one has 
	\[ \I_{v_\Sigma}^m = \bigcup_{\sigma \in \Sigma} \I_{v_\sigma}^m \ \text{and} \ \D_{v_\Sigma}^m = \bigcap_{\sigma \in \Sigma} \D_{v_\sigma}^m. \]
\end{lemma}
\begin{proof}
	Since the valuations $(v_\sigma)_{\sigma \in \Sigma}$ are pairwise comparable, their valuation theoretic supremum exists by general valuation theory.
	We let $v$ denote this supremum, and recall that 
	\[ \Oc_v = \bigcap_{\sigma \in \Sigma} \Oc_{v_\sigma}. \]
	Thus $v$ is the coarsest valuation such that $v_\sigma$ is a coarsening of $v$ for all $\sigma \in \Sigma$.

	Note that $\sigma \in \I_{v_\sigma}^m \subset \I_v^m$ for all $\sigma \in \Sigma$, and therefore $\Sigma \subset \I_v^m$.
	By Lemma \ref{lemma: valuative}, it follows that $v_\Sigma$ is a coarsening of $v$.
	However, $v_\sigma$ is a coarsening of $v_\Sigma$ for all $\sigma \in \Sigma$ by Lemma \ref{lemma: valuative}, thus $v = v_\Sigma$.

	Since $v$ is the supremum of $(v_\sigma)_{\sigma \in \Sigma}$, one has
	\[ \bigcap_{\sigma \in \Sigma} U_{v_\sigma} = U_{v} \ \text{and} \ \bigcup_{\sigma \in \Sigma} U_{v_\sigma}^1 = U_v^1. \]
	The fact that
	\[ \I_{v}^m = \bigcup_{\sigma \in \Sigma} \I_{v_\sigma}^m \ \text{and} \ \D_v^m = \bigcap_{\sigma \in \Sigma} \D_{v_\sigma}^m \]
	follows easily from this observation by using the definition of the minimized inertia and decomposition groups.
\end{proof}

\begin{lemma}
\label{lemma: non-valuative comparable}
	Let $m \in \Nbar$ be given, and let $K$ be a field.
	Let $v_1,v_2$ be two valuations of $K$, and assume that there exists some element $\sigma \in \D_{v_1}^m \cap \D_{v_2}^m$ such that $\sigma$ is non-valuative.
	Then the two valuations $v_1,v_2$ are comparable.
\end{lemma}
\begin{proof}
	Let $w$ denote the finest common coarsening of $v_1,v_2$ so that $\Oc_w = \Oc_{v_1} \cdot \Oc_{v_2}$, and let $w_i = v_i/w$ denote the valuation of $Kw$ induced by $v_i$ for $i = 1,2$.
	By basic valuation theory, if $w_1,w_2$ are both non-trivial, then they must be independent.

	If $w_1,w_2$ are indeed both non-trivial, then by the approximation theorem for independent valuations, one has $U_{w_1}^1 \cdot U_{w_2}^1 = Kw^\times$.
	Thus $\D_{w_1}^m \cap \D_{w_2}^m = 1$, so that $\D_{v_1}^m \cap \D_{v_2}^m = \I_w^m$ by Lemma \ref{lemma: residue C-pair}(2).
	But then $\sigma \in \I_w^m$, hence $\sigma$ is valuative.

	Since $\sigma$ is \emph{non-valuative} by assumption, we deduce that either $w_1$ or $w_2$ must be trivial.
	Hence $v_1$ and $v_2$ are comparable, as required.
\end{proof}

\subsection{C-pairs and Decomposition Elements}
Our final technical lemma essentially shows that elements which form a C-pair with a valuative element must arise from minimized decomposition.
\begin{lemma}
\label{lemma: C-pair implies decomposition}
	Let $m \in \Nbar$ and $M \geq \Mbf_1(m)$ be given, and let $K$ be a field.
	Let $\sigma' \in \gfrak^M(K)$ be a valuative element, and let $\tau' \in \gfrak^M(K)$ be such that $(\sigma',\tau')$ is a C-pair.
	Then $\sigma := \sigma'_m$ is valuative, and $\tau := \tau'_m \in \D_{v_\sigma}^m$.
\end{lemma}
\begin{proof}
	Since $\sigma'$ is valuative, the element $\sigma = \sigma'_m$ is also valuative.
	Put $v := v_\sigma$, and let $x \in \mf_v$ be given.
	We will show that $\tau(1-x) = \tau'_m(x) = 0$, which implies that $U_v^1 \subset \ker\tau$, hence $\tau \in \D_v^m$.

	\vskip 10pt
	\noindent\underline{\bf Case $\sigma(x) \neq 0$}.
	Note that $\sigma(1-x) = 0$ since $1-x \in U_v^1 \subset U_v$.
	Since $\sigma'$ is valuative as well, it follows that $\sigma'(1-x) \in \{0,\sigma'(x)\}$ by the ultrametric inequality, hence $\sigma'(1-x) = 0$ since $\sigma'(x) \neq 0$ by assumption.
	As $(\sigma',\tau')$ is a C-pair, we see that 
	\[ \sigma'(x)\tau'(1-x) = \sigma'(1-x)\tau'(x) = 0. \]
	But $\sigma'(x) \notin \ell^m\cdot \Lambda_M$ since $\sigma(x) \neq 0$.
	Therefore, $\tau(1-x) = 0$ by Fact \ref{fact: cancellation}.

	\vskip 10pt
	\noindent\underline{\bf Case $\sigma(x) = 0$}.
	Since $v(\ker(\sigma))$ contains no non-trivial convex subgroups by Lemma \ref{lemma: valuative}, there exists some $y \in \mf_v$ such that $\sigma(y) \neq 0$ and such that $0 < v(y) < v(x)$.
	One has $v(y+x\cdot (1-y)) = v(y)$ by the ultrametric inequality, and thus 
	\[ 0 \neq \sigma(y) = \sigma(y+x\cdot(1-y)). \]
	In particular, the first case implies that
	\[\tau(1-y) + \tau(1-x) = \tau((1-y)(1-x)) = \tau(1-(y+x\cdot(1-y))) = 0. \]
	Since $\tau(1-y) = 0$ as well by the first case, we see that $\tau(1-x) = 0$.
\end{proof}

\section{Proofs of Main Theorems}
\label{section: proof of main theorems}

\subsection{Preliminary Lemmas}
The proofs of our main theorems will all rely primarily on the following ``Key Lemma.''
\begin{lemma}[Key Lemma]
\label{lemma: key lemma}
	Let $n \in \Nbar$ and $N \geq \Rbf(n)$ be given, and put $m := \Mbf_1(n)$.
	Let $K$ be a field such that $\mu_{2\ell^N} \subset K$. 
	Let $\sigma,\tau_1,\tau_2 \in \gfrak^n(K)$ be given and let $\sigma',\tau_1',\tau_2'$ be $N$-lifts of $\sigma,\tau_1,\tau_2$.
	Assume that the following conditions hold true:
	\begin{enumerate}
		\item $(\tau_1,\tau_2)$ is not a C-pair.
		\item $(\sigma',\tau_1')$ and $(\sigma',\tau_2')$ are both C-pairs.
	\end{enumerate}
	Then $\sigma'_m$ is valuative, and thus $\sigma$ is valuative.
	Moreover, if $v := v_\sigma$ denotes the valuation associated to $\sigma$, then one has $\tau_1,\tau_2 \in \D_v^n$.
\end{lemma}
\begin{proof}
	For simplicity, we put $M = \Mbf_2(m) := \Mbf_2(\Mbf_1(n))$.
	Assume for a contradiction that $\sigma'_m$ is non-valuative.
	Then $\sigma'_M$ is also non-valuative.

	By condition (2) and Theorem \ref{theorem: main theorem of C-pairs}, we see that there exist valuations $v_1,v_2$ of $K$ such that the following conditions hold for $i = 1,2$:
	\begin{enumerate}
		\item One has $\sigma'_M,(\tau_i')_M \in \D_{v_i}^M$.
		\item The quotient $\langle \sigma'_M,(\tau_i')_M \rangle_{\Lambda_M} / (\langle \sigma'_M,(\tau_i')_M \rangle_{\Lambda_M} \cap \I_{v_i}^M)$ is $\Lambda_M$-cyclic.
	\end{enumerate}

	\begin{claim*}
		For $i = 1,2$, the subgroup $\langle \sigma'_M,(\tau_i')_M \rangle_{\Lambda_M}$ is non-$\Lambda_M$-cyclic.
	\end{claim*}
	\begin{proof}
		Fix $i \in \{1,2\}$, and assume for a contradiction that $\langle \sigma'_M,(\tau_i')_M \rangle$ is cyclic.
		Since $\Lambda_M$ is a quotient of a DVR, we see that either $\sigma'_M \in \Lambda_M \cdot (\tau_i')_M$ or $(\tau_i')_M \in \Lambda_M \cdot \sigma'_M$.
		If $(\tau_i')_M = a \cdot \sigma'_M$ for some $a \in \Lambda_M$, then $((\tau_1')_M,(\tau_2')_M)$ is a C-pair since $(\sigma'_M,(\tau_j')_M)$ is a C-pair for $j = 1,2$; this implies that $(\tau_1,\tau_2)$ is a C-pair, contradicting condition (1).
		
		On the other hand, suppose that $\sigma'_M = a \cdot (\tau_i')_M$ for some $a \in \Lambda_M$.
		Note that $a \notin \ell^m \cdot \Lambda_M$, for otherwise $\sigma'_m = 0$ would be a valuative element of $\gfrak^m(K)$.
		Since $(\sigma'_M,(\tau_j')_M)$ is a C-pair for $j = 1,2$, it follows from Fact \ref{fact: cancellation} that $(\tau_1,\tau_2)$ is a C-pair, again contradicting condition (1).
		The claim follows.
	\end{proof}

	By the claim above, we know that $\langle \sigma'_M,(\tau_i')_M \rangle$ is non-cyclic for $i = 1,2$.
	On the other hand, since $\langle \sigma'_M,(\tau_i')_M \rangle/ (\langle \sigma'_M,(\tau_i')_M \rangle_{\Lambda_M}\cap \I_{v_i}^M)$ is cyclic, we see that there exist $(a_i,b_i) \in \Lambda_M^2 \smallsetminus \ell \cdot \Lambda_M^2$ such that 
	\[ a_i \cdot \sigma'_M + b_i \cdot (\tau_i')_M \in \I_{v_i}^M. \]

	Since $\sigma'_M \in \D_{v_1}^M \cap \D_{v_2}^M$, and $\sigma'_M$ is non-valuative by assumption, it follows from Lemma \ref{lemma: non-valuative comparable} that $v_1$ and $v_2$ are comparable.
	Without loss of generality, we may assume that $v_1$ is coarser than $v_2$.
	Therefore, $\I_{v_1}^M \subset \I_{v_2}^M$.
	Since $\sigma'_M \in \D_{v_2}^M$, it follows from Lemma \ref{lemma: residue C-pair}(3) that 
	\[ \langle \sigma'_M, \ a_1 \cdot \sigma'_M + b_1 \cdot (\tau_1')_M, \ a_2 \cdot \sigma'_M + b_2 \cdot (\tau_2')_M \rangle_{\Lambda_M} = \langle \sigma'_M, \ b_1\cdot (\tau_1')_M, \ b_2\cdot (\tau_2')_M \rangle_{\Lambda_M} \]
	is a C-set.
	In particular, $(b_1\cdot (\tau_1')_M,b_2\cdot (\tau_2')_M)$ is a C-pair.

	To conclude the proof of the lemma, we first note that $b_i \notin \ell^m \cdot \Lambda_M$ for $i = 1,2$.
	Indeed, if $b_i \in \ell^m \cdot \Lambda_M$, then the element $(a_i \cdot \sigma')_m = (a_i \cdot \sigma'_M + b_i \cdot (\tau_i')_M)_m \in \I_{v_i}^m$ is valuative.
	But, if $b_i \in \ell^m \cdot \Lambda_M$ then $a_i$ must be a unit in $\Lambda_M$ since $(a_i,b_i) \notin \ell \cdot \Lambda_M^2$.
	But this would imply that $\sigma'_m$ is valuative, hence contradicting our original assumption.

	Since $(b_1\cdot (\tau_1')_M, \ b_2\cdot (\tau_2')_M)$ is a C-pair, while $b_1,b_2 \notin \ell^m \cdot \Lambda_M$, we deduce from Fact \ref{fact: cancellation} that $((\tau_1')_m,(\tau_2')_m)$ is a C-pair in $\gfrak^m(K)$.
	Thus $(\tau_1,\tau_2)$ is a C-pair as well, which contradicts condition (1) of the lemma.
	Thus $\sigma'_m$ is valuative, hence $\sigma$ is valuative as well.
	Finally, the fact that $\tau_1,\tau_2 \in \D_{v_\sigma}^n$ follows from Lemma \ref{lemma: C-pair implies decomposition}.
\end{proof}

We will also need to reduce some arguments/constructions to the case $n = 1$, which will be accomplished using the following two lemmas.
Since it will be used several times in these two lemmas, we recall from Remark \ref{remark: centers centralizers} that
\[ \I^1_1(\gfrak^1(K)) = \{\sigma \in \gfrak^1(K) \ : \ \text{For all $\tau \in \gfrak^1(K)$, $(\sigma,\tau)$ is a C-pair}\}\]
is the ``C-center'' of $\gfrak^1(K)$.
\begin{lemma}
\label{lemma: n equals 1}
	Let $m \in \Nbar$ be given, and let $K$ be a field such that $\mu_{2\ell^m} \subset K$.
	Then the following hold:
	\begin{enumerate}
		\item Suppose that $\I^1_1(\gfrak^1(K)) \neq \gfrak^1(K)$. 
		Then there exists a $1$-visible valuation $v$ of $K$ such that $\I^1_1(\gfrak^1(K)) = \I_v^1$ and $\gfrak^1(K) = \D_v^1$.
		\item Suppose that $\I^1_1(\gfrak^1(K)) = \gfrak^1(K)$. 
		Then there exists a valuation $v$ of $K$ such that $\D_v^1 = \gfrak^1(K)$ and such that $\gfrak^1(Kv)$ is cyclic.
		Moreover, in this case $\gfrak^m(K)$ is a C-set.
	\end{enumerate}
\end{lemma}
\begin{proof}
	\noindent\emph{Proof of (1).}
	Put $\I = \I^1_1(\gfrak^1(K))$.
	Since $\I \neq \gfrak^1(K)$, it follows from the definition of $\I^1_1$ that $\gfrak^1(K)$ is not a C-set.
	As such, let $\tau_1,\tau_2 \in \gfrak^1(K)$ be two elements such that $(\tau_1,\tau_2)$ is not a C-pair.
	By Lemma \ref{lemma: key lemma}, it follows that every element of $\I$ is valuative.
	Moreover, for every element $\tau \in \gfrak^1(K)$ and every $\sigma \in \I$, one has $\tau \in \D_{v_\sigma}^1$ by Lemma \ref{lemma: C-pair implies decomposition}.
	Hence $\D_{v_\sigma}^1 = \gfrak^1(K)$ for all $\sigma \in \I$.

	On the other hand, for every $\sigma,\tau \in \I$, the pair $(\sigma,\tau)$ is a C-pair by the definition of $\I$.
	Thus, by Lemma \ref{lemma: C-pair implies comparable}, it follows that the valuations $(v_\sigma)_{\sigma \in \I}$ are pairwise comparable.
	Hence $\I$ is valuative by Lemma \ref{lemma: supremum}; we put $v := v_{\I}$.
	By Lemma \ref{lemma: supremum} we deduce that $\gfrak^1(K) = \D_v^1$.

	Recall that $\I \subset \I_v^1$ by the definition of $v_{\I}$.
	On other hand, if $\sigma \in \I_v^1$ and $\tau \in \gfrak^1(K) = \D_v^1$, then by Lemma \ref{lemma: residue C-pair}(3) we see that $(\sigma,\tau)$ is a C-pair.
	Hence $\sigma \in \I$ by the definition of $\I^1_1$.
	Namely, one has $\I = \I_v^1$.

	To conclude the proof of (1), we must show that $v$ is $1$-visible.
	Since $v = v_{\I}$, it follows from Lemma \ref{lemma: valuative} that $vK$ contains no non-trivial $\ell$-divisible convex subgroups.
	Also, we know that $\gfrak^1(Kv)$ is not a C-set for otherwise $\gfrak^1(K) = \D_v^1$ would be a C-set by Lemma \ref{lemma: residue C-pair}(3).

	Finally, suppose that $w_0$ is a valuation of $Kv$ such that $\D_{w_0}^1 = \gfrak^1(Kv)$.
	Consider $w := w_0 \circ v$, and note that $\D_w^1 = \gfrak^1(K)$ by Lemma \ref{lemma: residue C-pair}(2).
	But then Lemma \ref{lemma: residue C-pair}(3) implies that $\I_w^1 \subset \I$, by the definition of $\I^1_1$, similarly to the argument above which shows that $\I_v^1 \subset \I$.
	On the other hand $\I = \I_v^1 \subset \I_w^1$ since $v$ is coarser than $w$.
	Thus $\I_v^1 = \I_w^1$, and therefore $\I_{w_0}^1 = 1$ by Lemma \ref{lemma: residue C-pair}(2).

	\vskip 10pt
	\noindent\emph{Proof of (2).}
	The condition $\I^1_1(\gfrak^1(K)) = \gfrak^1(K)$ is equivalent to saying that $\gfrak^1(K)$ is a C-set.
	Let $\Sigma$ denote the subset of $\gfrak^1(K)$ consisting of all valuative elements of $\gfrak^1(K)$.
	By Theorem \ref{theorem: main theorem of C-pairs}, it follows that $\gfrak^1(K)/\langle \Sigma \rangle_{\Lambda_1}$ is cyclic.
	Moreover, the valuations $(v_\sigma)_{\sigma \in \Sigma}$ are pairwise comparable by Lemma \ref{lemma: C-pair implies comparable}.
	Thus, by Lemma \ref{lemma: supremum}, we deduce that $\Sigma$ is itself valuative, hence $\Sigma = \langle \Sigma \rangle_{\Lambda_1}$ by the way we defined $\Sigma$.

	On the other hand, for all $\sigma \in \Sigma$, it follows from Lemma \ref{lemma: C-pair implies decomposition} that $\gfrak^1(K) = \D_{v_\sigma}^1$.
	Letting $v := v_\Sigma$ denote the valuation associated to $\Sigma$, we deduce that $\D_v^1 = \gfrak^1(K)$ by Lemma \ref{lemma: supremum}.
	Moreover, $\D_v^1/\I_v^1$ is cyclic since $\Sigma \subset \I_v^1$.
	This implies that $\gfrak^1(Kv)$ is cyclic by Lemma \ref{lemma: residue C-pair}(1).

	To conclude the proof of the Lemma, we must prove that $\gfrak^m(K)$ is a C-set.
	By Lemma \ref{lemma: m-lifts}(5), we have $\D_v^m = \gfrak^m(K)$ and by Lemma \ref{lemma: m-lifts}(2), $\gfrak^m(Kv)$ is $\Lambda_m$-cyclic.
	Thus, by Lemma \ref{lemma: residue C-pair}(1), the quotient $\D_v^m/\I_v^m$ is $\Lambda_m$-cyclic.
	This implies that $\gfrak^m(K)$ is a C-set by Lemma \ref{lemma: residue C-pair}(3).
	This concludes the proof of the lemma.
\end{proof}

\begin{lemma}
\label{lemma: visible reduction to n equals 1}
	Let $m \in \Nbar$ be given, and let $K$ be a field such that $\mu_{2\ell^m} \subset K$.
	Let $v$ be a valuation of $K$.
	Then the following hold:
	\begin{enumerate}
		\item Suppose that $vK$ contains no non-trivial $\ell$-divisible convex subgroups, and that $w_0$ is an $m$-visible valuation of $Kv$.
		Then $w_0 \circ v$ is an $m$-visible valuation of $K$.
		\item If $v$ is $1$-visible then $v$ is $m$-visible.
	\end{enumerate}
\end{lemma}
\begin{proof}
	\noindent\emph{Proof of (1).}
	Put $w := w_0 \circ v$.
	Since $vK$ and $w_0(Kv)$ contain no non-trivial $\ell$-divisible convex subgroups, the same must be true for $wK$ by considering the short exact sequence
	\[ 1 \rightarrow w_0(Kv) \rightarrow wK \rightarrow vK \rightarrow 1. \]
	The other two conditions required for $w$ to be $m$-visible are clear since the residue field of $w$ is the same as the residue field of $w_0$.

	\vskip 10pt
	\noindent\emph{Proof of (2).}
	Suppose that $v$ is $1$-visible.
	Then $vK$ has no $\ell$-divisible convex subgroups.
	Also, since $\gfrak^1(Kv)$ is not a C-set, it follows from Lemma \ref{lemma: m-lifts}(1) that $\gfrak^m(Kv)$ is not a C-set either.
	Finally, suppose that $w$ is a valuation of $Kv$ such that $\gfrak^m(Kv) = \D_w^m$.
	Then $\gfrak^1(Kv) = \D_w^1$ hence $\I_w^1 = 1$ since $v$ is $1$-visible.
	But then $\I_w^m = 1$ by Lemma \ref{lemma: m-lifts}(4).
\end{proof}

\subsection{Proof of Theorem \ref{maintheorem: definability of visible inertia set}}
We now turn to the proof of Theorem \ref{maintheorem: definability of visible inertia set}, and we use the notation introduced in the statement of the theorem.

\vskip 10pt
\noindent$(1) \Rightarrow (2)$.
Let $v$ be an $n$-visible valuation of $K$ such that $\sigma \in \I_v^n$.
By condition (V2), we know that $\gfrak^n(Kv)$ is not a C-set, and thus by Lemma \ref{lemma: residue C-pair}(3) we deduce that $\D_v^n$ is not a C-set.
Let $\tau_1,\tau_2 \in \D_v^n$ be two elements such that $(\tau_1,\tau_2)$ is not a C-pair.

By Lemma \ref{lemma: m-lifts}(3), we can choose $N$-lifts $\sigma',\tau_1',\tau_2'$ of $\sigma,\tau_1,\tau_2$ such that $\sigma' \in \I_v^N$ and $\tau_1',\tau_2' \in \D_v^N$.
Finally, by Lemma \ref{lemma: residue C-pair}(3), we see that $(\sigma',\tau_1')$ and $(\sigma',\tau_2')$ are both C-pairs, as required.

\vskip 10pt
\noindent$(2)\Rightarrow(1)$.
By Lemma \ref{lemma: key lemma}, we know that $\sigma$ is valuative, and that $\tau_1,\tau_2 \in \D_{v_\sigma}^n$.
We will show that $v_\sigma$ is a coarsening of an $n$-visible valuation $v$, which means that 
\[ \sigma \in \I_{v_\sigma}^n \subset \I_v^n \]
and therefore $\sigma \in \Ivis^n(K)$.

Consider $\I := \I^1_1(\gfrak^1(Kv_\sigma))$.
If $\I = \gfrak^1(Kv_\sigma)$, then Lemma \ref{lemma: n equals 1}(2) implies that $\gfrak^n(Kv)$ is a C-set, hence $\D_{v_\sigma}^n$ is a C-set by Lemma \ref{lemma: residue C-pair}(3).
But this contradicts the fact that $\tau_1,\tau_2 \in \D_{v_\sigma}^n$ and $(\tau_1,\tau_2)$ is not a C-pair.

Thus $\I \neq \gfrak^1(Kv_\sigma)$.
By Lemma \ref{lemma: n equals 1}(1), there exists a $1$-visible valuation $w_0$ of $Kv_\sigma$.
But by Lemma \ref{lemma: visible reduction to n equals 1}(2), we see that $w_0$ is $n$-visible.
Finally, by Lemma \ref{lemma: visible reduction to n equals 1}(1), we deduce that $v := w_0 \circ v_\sigma$ is $n$-visible, since $v_\sigma K$ contains no non-trivial $\ell$-divisible convex subgroups by Lemma \ref{lemma: valuative}.
This concludes the proof of Theorem \ref{maintheorem: definability of visible inertia set}.

\subsection{Proof of Theorem \ref{maintheorem: definability of visible inertia groups}}

We now turn to the proof of Theorem \ref{maintheorem: definability of visible inertia groups}, and we use the notation introduced in the statement of the theorem.

\vskip 10pt
\noindent$(1) \Rightarrow (2)$.
Suppose that $v$ is an $n$-visible valuation of $K$ such that $\Sigma \subset \I_v^n$.
By condition (V2), we know that $\gfrak^n(Kv)$ is not a C-set, so Lemma \ref{lemma: residue C-pair}(3) implies that $\D_v^n$ is not a C-set.
Let $\tau_1,\tau_2 \in \D_v^n$ be two elements such that $(\tau_1,\tau_2)$ is not a C-pair.

Let $\sigma,\tau \in \Sigma$ be given.
By Lemma \ref{lemma: m-lifts}(3), there exist $N$-lifts $\sigma',\tau',\tau_1',\tau_2'$ of $\sigma,\tau,\tau_1,\tau_2$ such that $\sigma',\tau' \in \I_v^N$ and $\tau_1',\tau_2' \in \D_v^N$.
By Lemma \ref{lemma: residue C-pair}(3), we deduce that the following conditions hold true:
\begin{enumerate}
	\item $(\sigma',\tau')$ is a C-pair.
	\item $(\sigma',\tau_1')$ and $(\sigma',\tau_2')$ are both C-pairs.
\end{enumerate}

\vskip 10pt
\noindent$(2) \Rightarrow (1)$.
First of all, conditions (b),(c) and Theorem \ref{maintheorem: definability of visible inertia set} imply that $\Sigma \subset \Ivis^n(K)$.
Thus every element of $\Sigma$ is valuative.

By condition (a) and Lemma \ref{lemma: C-pair implies comparable}, we see that the valuations in the collection $(v_\sigma)_{\sigma \in \Sigma}$ are pairwise comparable.
By Lemma \ref{lemma: supremum}, we see that $\Sigma$ is valuative, and that 
\[ \D_{v_\Sigma}^n = \bigcap_{\sigma \in \Sigma} \D_{v_\sigma}^n. \]
Moreover, by Lemma \ref{lemma: key lemma} and conditions (b) and (c), we see that $\tau_1,\tau_2 \in \D_{v_\sigma}^n$ for all $\sigma \in \Sigma$.
Therefore $\tau_1,\tau_2 \in \D_{v_\Sigma}^n$ by Lemma \ref{lemma: supremum}.

Now consider $\I := \I^1_1(\gfrak^1(Kv_\Sigma))$.
If $\I = \gfrak^1(Kv_\Sigma)$ then Lemma \ref{lemma: n equals 1}(2) implies that $\gfrak^n(Kv_\Sigma)$ is a C-set, and therefore Lemma \ref{lemma: residue C-pair}(3) implies that $\D_{v_\Sigma}^n$ is a C-set.
But this contradicts the fact that $\tau_1,\tau_2 \in \D_{v_\Sigma}^n$ and $(\tau_1,\tau_2)$ is not a C-pair.

Thus $\I \neq \gfrak^1(Kv_\Sigma)$.
By Lemma \ref{lemma: n equals 1}(1), we see that there exists a $1$-visible valuation $w_0$ of $Kv_\Sigma$.
By Lemma \ref{lemma: visible reduction to n equals 1}(2), we see that $w_0$ is $n$-visible, and by Lemma \ref{lemma: visible reduction to n equals 1}(1), we deduce that $v := w_0 \circ v_\Sigma$ is $n$-visible since $v_\Sigma K$ contains no non-trivial $\ell$-divisible convex subgroups by Lemma \ref{lemma: valuative}.
Therefore, one has $\Sigma \subset \I_{v_\Sigma}^n \subset \I_v^n$, with $v$ an $n$-visible valuation.
This concludes the proof of Theorem \ref{maintheorem: definability of visible inertia groups}.

\subsection{Proof of Theorem \ref{maintheorem: definability of visible decomposition groups}}
We now turn to the proof of Theorem \ref{maintheorem: definability of visible decomposition groups}, and we use the notation introduced in the statement of the theorem.

\vskip 10pt
\noindent\emph{Proof of (1).}
First of all, we note that $\Sigma$ is valuative by assumption.
We put $v_0 := v_\Sigma$.
By Lemma \ref{lemma: valuative}, we see that $v_0$ is a coarsening of some $n$-visible valuation $v_1$.
Thus $\gfrak^n(Kv_1)$ is not a C-set, so Lemma \ref{lemma: residue C-pair}(3) implies that $\D_{v_1}^n$ is not a C-set.
Hence $\D_{v_0}^n$ is not a C-set, since $\D_{v_1}^n \subset \D_{v_0}^n$.

\begin{claim*}
	One has $\D_{v_0}^n = \D^N_n(\Sigma)$ and $\Sigma \subset \I^N_n(\D^N_n(\Sigma))$.
\end{claim*}
\begin{proof}
	Let $\tau \in \D_{v_0}^n$ be given, and suppose that $\sigma \in \Sigma$ is an arbitrary element.
	By the observation above, there exist $\tau_1,\tau_2 \in \D_{v_0}^n$ such that $(\tau_1,\tau_2)$ is not a C-pair.
	On the other hand, by Lemma \ref{lemma: m-lifts}(3), there exist $N$-lifts $\sigma',\tau',\tau_1',\tau_2'$ of $\sigma,\tau,\tau_1,\tau_2$ such that $\sigma' \in \I_{v_0}^N$ and $\tau',\tau_1',\tau_2' \in \D_{v_0}^N$.
	By Lemma \ref{lemma: residue C-pair}(3), we see that the three pairs $(\sigma',\tau')$, $(\sigma',\tau_1')$ and $(\sigma',\tau_2')$ are all C-pairs.
	Thus $\tau \in \D^N_n(\Sigma)$.

	Conversely, suppose that $\tau \in \D^N_n(\Sigma)$ is given.
	Let $\sigma \in \Sigma$ be an element, and let $\tau_1,\tau_2$ and $\sigma',\tau',\tau_1',\tau_2'$ be as in the definition of $\D^N_n(\Sigma)$.
	Put $m = \Mbf_1(n)$.
	By Lemma \ref{lemma: key lemma}, we see that $\sigma'_m$ is valuative.
	Thus, by Lemma \ref{lemma: C-pair implies decomposition} we deduce that $\tau$ is an element of $\D_{v_\sigma}^n$.
	But this is true for all $\sigma \in \Sigma$, and thus
	\[ \D^N_n(\Sigma) \subset \bigcap_{\sigma \in \Sigma} \D_{v_\sigma}^n. \]
	We deduce that $\D^N_n(\Sigma) \subset \D_{v_0}^n$ by Lemma \ref{lemma: supremum}.
	Hence $\D^N_n(\Sigma) = \D_{v_0}^n$.

	Now suppose that $\sigma \in \Sigma$ is given.
	Then by Lemma \ref{lemma: m-lifts}(3), there exists an $N$-lift $\sigma'$ of $\sigma$ such that $\sigma' \in \I_{v_0}^N$.
	Also, by Lemma \ref{lemma: m-lifts}(3), for every $\tau \in \D_{v_0}^n = \D^N_n(\Sigma)$, there exists some $N$-lift $\tau'$ of $\tau$ such that $\tau' \in \D_{v_0}^N$.
	By Lemma \ref{lemma: residue C-pair}(3), we see that $(\sigma',\tau')$ is a C-pair, thus $\sigma \in \I^N_n(\D_{v_0}^n)$.
	Hence, $\Sigma \subset \I^N_n(\D_{v_0}^n) = \I^N_n(\D^N_n(\Sigma))$.
\end{proof}

To conclude the proof of the first part of the theorem, we must prove that there exists an $n$-visible valuation $v$ such that $\D_v^n = \D_{v_0}^n$ and such that $\I_v^n = \I^N_n(\D_{v_0}^n)$.
For simplicity, let $\D$ denote the set $\D_{v_0}^n = \D^N_n(\Sigma)$, and let $\I$ denote the set $\I^N_n(\D)$.
First we will show that $\I$ is valuative.

Let $\sigma \in \I$ be given.
Since $\D$ is not a C-set, there exist $\tau_1,\tau_2 \in \D$ such that $(\tau_1,\tau_2)$ is not a C-pair.
On the other hand, by the definition of $\I$, there exists an $N$-lift $\sigma'$ of $\sigma$ and $N$-lifts $\tau_1',\tau_2'$ of $\tau_1,\tau_2$ such that $(\sigma',\tau_1')$ and $(\sigma',\tau_2')$ are both C-pairs.
By Lemma \ref{lemma: key lemma}, we deduce that $\sigma$ is valuative.

On the other hand, if $\sigma,\tau \in \I$, then using the definition of $\I$ again, we see that there exist $N$-lifts $\sigma',\tau'$ of $\sigma,\tau$ such that $(\sigma',\tau')$ is a C-pair.
Thus, by Lemma \ref{lemma: C-pair implies comparable}, we see that the collection of valuations $(v_\sigma)_{\sigma \in \I}$ is pair-wise comparable.
By Lemma \ref{lemma: supremum}, we deduce that $\I$ is valuative and that $v := v_{\I}$ is the valuation-theoretic supremum of the collection $(v_\sigma)_{\sigma \in \I}$.

\begin{claim*}
	One has $\D = \D_v^n$.	
\end{claim*}
\begin{proof}
	By Lemma \ref{lemma: valuative} and the fact that $\Sigma \subset \I$, we see that $v_0 = v_\Sigma$ is a coarsening of $v = v_{\I}$.
	Thus, it suffices to prove that $\D \subset \D_v^n$ since the inclusion $\D_v^n \subset \D$ is already known.
	As such, suppose that $\tau \in \D$ is given, and let $\tau_1,\tau_2 \in \D$ be two elements such that $(\tau_1,\tau_2)$ is not a C-pair, as above.
	For $\sigma \in \I$, it follows from the definition of $\I$ that there exist $N$-lifts $\sigma',\tau',\tau_1',\tau_2'$ of $\sigma,\tau,\tau_1,\tau_2$ such that the three pairs $(\sigma',\tau')$, $(\sigma',\tau_1')$ and $(\sigma',\tau_2')$ are all C-pairs.
	By Lemma \ref{lemma: key lemma} we see that $\sigma'_m$ is valuative, and thus by Lemma \ref{lemma: C-pair implies decomposition} we deduce that $\tau \in \D_{v_\sigma}^n$.
	Since this is true for all $\sigma \in \I$, it follows from Lemma \ref{lemma: supremum} that $\tau \in \D_v^n$.
	Thus $\D \subset \D_v^n$ as required.
\end{proof}

\begin{claim*}
	One has $\I = \I_v^n$.
\end{claim*}
\begin{proof}
	We already know that $\I \subset \I_v^n$ by the definition of $v = v_{\I}$.
	Suppose that $\sigma \in \I_v^n$ is given.
	By Lemma \ref{lemma: m-lifts}(3), there exists an $N$-lift $\sigma'$ of $\sigma$ such that $\sigma' \in \I_v^N$.
	On the other hand, by Lemma \ref{lemma: m-lifts}(3), we know that for all $\tau \in \D = \D_v^n$, there exists an $N$-lift $\tau'$ of $\tau$ such that $\tau' \in \D_v^N$.
	With this notation, it follows from Lemma \ref{lemma: residue C-pair}(3) that $(\sigma',\tau')$ is a C-pair.
	Thus $\sigma \in \I$ by the definition of $\I$.
\end{proof}

To conclude the proof of (1), we must prove that $v$ is an $n$-visible valuation.
First, since $v = v_{\I}$, it follows from Lemma \ref{lemma: valuative} that condition (V1) holds true.
Also, as noted above, there exist $\tau_1,\tau_2 \in \D_v^n$ such that $(\tau_1,\tau_2)$ is not a C-pair.
Thus condition (V2) holds true by Lemma \ref{lemma: residue C-pair}(3).
Finally, suppose that $w_0$ is a valuation of $Kv$ such that $\D_w = \gfrak^n(Kv)$, and put $w = w_0 \circ v$.
Then by Lemma \ref{lemma: residue C-pair}(2), we see that $\D_w^n = \D_v^n$, while $\I_v^n \subset \I_w^n$.

For every $\sigma \in \I_w^n$, by Lemma \ref{lemma: m-lifts}(3), there exists some $N$-lift $\sigma'$ of $\sigma$ such that $\sigma' \in \I_w^N$.
By Lemma \ref{lemma: m-lifts}(3), for every $\tau \in \D_w^n = \D$, there exists an $N$-lift $\tau'$ of $\tau$ such that $\tau' \in \D_w^N$.
But then by Lemma \ref{lemma: residue C-pair}(3) it follows that $(\sigma',\tau')$ is a C-pair.
Therefore $\I_w^n \subset \I$ by the definition of $\I$, while $\I = \I_v^n \subset \I_w^n$.
In particular $\I_w^n = \I_v^n$ and therefore $\I_{w_0}^n = 0$ by Lemma \ref{lemma: residue C-pair}(2).
This shows that $v$ satisfies condition (V3).

\vskip 10pt
\noindent\emph{Proof of (2).}
Let $v$ be an $n$-visible valuation, and let $\Sigma \subset \I_v^n$ a subset such that $v(\Sigma^\perp)$ contains no non-trivial convex subgroups.
We must show that $\D_v^n = \D^N_n(\Sigma)$ and that $\I_v^n = \I^N_v(\D^N_n(\Sigma))$.

Following exactly the argument of (1) above with the given $\Sigma$, we see that $\D := \D^N_n(\Sigma) = \D_{v_0}^n$ where $v_0 = v_\Sigma$.
But by Lemma \ref{lemma: valuative}, we see that $v = v_\Sigma$, so that $v_0 = v$.

Hence, it suffices to prove that $\I := \I^N_n(\D) = \I_v^n$.
By Lemma \ref{lemma: m-lifts}(3), for every $\sigma \in \I_v^n$, there exists an $N$-lift $\sigma'$ of $\sigma$ such that $\sigma' \in \I_v^N$.
By Lemma \ref{lemma: m-lifts}(3) again, for every $\tau \in \D = \D_v^n$, there exists an $N$-lift $\tau'$ of $\tau$ such that $\tau' \in \D_v^N$.
By Lemma \ref{lemma: residue C-pair}(3) we see that $(\sigma',\tau')$ is a C-pair.
Therefore, one has $\I_v^n \subset \I$ by the definition of $\I$.

Finally, by assertion (1) above, we see that $\I$ is valuative, and letting $w := v_{\I}$, one has $\D_w^n = \D$ and $\I_w^n = \I$.
Let $w_0 = w/v$ denote the valuation of $Kv$ induced by $w$.
Then by Lemma \ref{lemma: residue C-pair}(2), we see that one has $\D_{w_0}^n = \gfrak^n(Kv)$.
Since $v$ is $n$-visible, we deduce that $\I_{w_0}^n = 1$ and thus $\I_w^n = \I_v^n$ by Lemma \ref{lemma: residue C-pair}(2).
Finally, Lemma \ref{lemma: valuative} implies that $v = w$.
Therefore $\I = \I_v^n$ and $\D = \D_v^n$.
This concludes the proof of Theorem \ref{maintheorem: definability of visible decomposition groups}.

\section{Quasi-Divisorial Valuations}
\label{section: quasi-divisorial valuations}

In this section we recall the necessary facts concerning quasi-divisorial valuations of function fields.
This terminology was introduced by {\sc Pop} \cite{Pop2010}, and we will refer to Remark/Definition 4.1 of loc.cit. and \cite[Facts 5.4, 5.5]{Pop2006} for the various general statements concerning (almost $r$-)quasi-divisorial valuations.

Throughout this section, $K$ will be a function field over an algebraically closed field $k$.
Let $v$ be a valuation of $K$.
Recall that Abhyankar's inequality states that 
\[ \rr(vK/vk) + \trdeg(Kv|kv) \leq \trdeg(K|k) \]
where $\rr(vK/vk) := \dim_\Q((vK/vk)\otimes_\Z \Q)$ denotes the {\bf rational rank} of $vK/vk$.
We say that $v$ {\bf has no transcendence defect} if the above inequality is an equality.
The following fact is more-or-less well known; see \cite[Facts 5.4]{Pop2006} for a precise reference.

\begin{fact}
\label{fact: defectless coarsening}
	In the context above, suppose that $w$ is a valuation of $K$ which has no transcendence defect, and let $v$ be a coarsening of $w$.
	Then $v$ has no transcendence defect.
\end{fact}

\subsection{almost-$r$-quasi-divisorial valuations}

Let $v$ be a valuation of $K$ and let $r$ be such that $1 \leq r \leq \trdeg(K|k) =: d$.
Following \cite{Pop2011a}, we say that $v$ is an {\bf almost-$r$-quasi-divisorial valuation} of $K|k$ if the following conditions hold true:
\begin{enumerate}
	\item $vK$ contains no non-trivial $\ell$-divisible convex subgroups.
	\item $\rr(vK/vk) = r$.
	\item $v$ has no transcendence defect, i.e. $\trdeg(K|k)-r = \trdeg(Kv|kv)$.
\end{enumerate}

Valuations which are almost-$r$-quasi-divisorial always exist, as follows.
Suppose that $X$ is a normal model for $K|k$.
Let $v$ be the discrete rank $r$ valuation of $K$ defined by a flag of Weil-prime-divisors of length $r$ on $X$.
Then it immediately follows from the definition that $v$ is an almost-$r$-quasi-divisorial valuation of $K|k$.
Thus, almost-$r$-quasi-divisorial valuations always exist for all $r$ such that $1 \leq r \leq \trdeg(K|k)$.
In general, however, there are many almost-$r$-quasi-divisorial valuations of $K|k$ which are \emph{non-trivial on $k$}.

The following fact summarizes the other basic required facts concerning almost-$r$-quasi-divisorial valuations.
Again, refer to {\sc Pop} \cite[Remark/Definition 4.1]{Pop2010} and \cite[Facts 5.4, 5.5]{Pop2006} for the proofs of these statements.

\begin{fact}
\label{fact: r-quasi-divisorial facts}
	In the context above, suppose that $v$ is an almost-$r$-quasi-divisorial valuation of $K|k$.
	Then the following hold:
	\begin{enumerate}
		\item $vK/vk \cong \Z^r$ as abstract groups.
		\item $Kv$ is a function field of transcendence degree $\trdeg(K|k)-r$ over $kv$.
		\item If $v_0$ is an almost-$r_0$-quasi-divisorial valuation of $Kv|kv$, then $v_0 \circ v$ is an almost-$(r+r_0)$-quasi-divisorial valuation of $K|k$.
	\end{enumerate}
\end{fact}
By Fact \ref{fact: r-quasi-divisorial facts}(1), we see that the definition of a {\bf quasi-divisorial valuation} from \S\ref{subsection: intro / quasi-divisorial valuations} agrees with the definition of an almost-$1$-quasi-divisorial valuation.
Moreover, the following lemma shows that almost-$r$-quasi-divisorial valuations are $m$-visible for all $m$, provided that $r$ is smaller than $\trdeg(K|k)$.

\begin{lemma}
\label{lemma: quasi-divisors are visible}
	Let $m \in \Nbar$ be given.
	Let $K$ be a function field over an algebraically closed field $k$ and let $r$ be such that $1 \leq r < \trdeg(K|k)$.
	Let $v$ be an almost-$r$-quasi-divisorial valuation of $K$.
	Then $v$ is $m$-visible.
\end{lemma}
\begin{proof}
	By Lemma \ref{lemma: visible reduction to n equals 1}(2), it suffices to prove that $v$ is $1$-visible.
	Condition (V1) is part of the definition of an almost-$r$-quasi-divisorial valuation.
	Furthermore, by Fact \ref{fact: r-quasi-divisorial facts}(2), the residue field $Kv$ is a function field of transcendence degree $\geq 1$ over $kv$.
	
	By \cite[Example 4.3]{Topaz2012c}, it follows (using the notation of loc.cit.) that $v \in \Vc_{K,1}$.
	Next, by \cite[Lemma 4.8]{Topaz2012c}, it follows (using our notation) that $\I_v^1 = \I^1_1(\D_v^1) \neq \D_v^1$.
	Thus, by Lemma \ref{lemma: residue C-pair}(3) and the definition of $\I^1_1$, we deduce that $\gfrak^1(Kv)$ is not a C-set.

	Finally, suppose that $w_0$ is a valuation of $Kv$ such that $\D_{w_0}^1 = \gfrak^1(Kv)$, and put $w := w_0 \circ v$.
	Then by Lemma \ref{lemma: residue C-pair}(2) we deduce that $\D_w^1 = \D_v^1$. 
	Thus, by Lemma \ref{lemma: residue C-pair}(3) and the definition of $\I^1_1$, it follows that $\I_w^1 \subset \I^1_1(\D_v^1) = \I_v^1$.
	Since $v$ is a coarsening of $w$, we deduce that $\I_v^1 = \I_w^1$, hence $\I_{w_0}^1 = 1$ by Lemma \ref{lemma: residue C-pair}(2).
\end{proof}

\subsection{$\ell$-rank}
We will need to use a concept which is similar to the rational rank of $vK/vk$.
For an arbitrary valuation $v$ of $K$, we define
\[ \rl(vK) = \dim_{\Z/\ell}(vK \otimes_\Z \Z/\ell) \]
and call $\rl(vK)$ the {\bf $\ell$-rank of $vK$.}
By Pontryagin duality, we immediately see that $\rl(vK)$ is the same as the rank of $\I_v^n$ as a pro-$\ell$ group.
Moreover, since $k$ is algebraically closed, hence $vk$ is divisible and $vK/vk$ is torsion-free, it follows that 
\[ \rl(vK) \leq \rr(vK/vk). \]
In particular, we obtain an inequality involving $\ell$-rank which is analogous to Abhyankar's inequality:
\[ \rl(vK) + \trdeg(Kv|kv) \leq \trdeg(K|k). \]
Furthermore, if this inequality is an \emph{equality}, then $v$ has no transcendence defect.

\subsection{Reconstructing inertia/decomposition}
We now prove our main theorem concerning the minimized inertia/decomposition groups of quasi-divisorial valuations of $K|k$.
First, we show how to recover the transcendence degree.
The argument for Theorem \ref{theorem: transcendence degree} is similar to \cite{Pop2010} resp. \cite{Pop2011a} where a similar statement is proven for $n=\infty$ resp. $n = 1$.
\begin{theorem}
\label{theorem: transcendence degree}
	Let $n \in \Nbar$ be given.
	Let $K$ be a function field over an algebraically closed field $k$.
	Then $d := \trdeg(K|k)$ is maximal among the non-negative integers $r$ such that the following holds true:
	There exist $\sigma_1,\ldots,\sigma_r \in \gfrak^n(K)$ such that $\sigma_1,\ldots,\sigma_r$ are $\Lambda_n$-independent and $\{\sigma_1,\ldots,\sigma_r\}$ is a C-set.
\end{theorem}
\begin{proof}
	First suppose that $r = d$.
	Let $v$ be any almost-$d$-quasi-prime-divisor of $K|k$; e.g. $v$ can be taken to be the valuation associated to a flag of Weil-prime-divisors of maximal length on some normal model of $K|k$.
	Then $vK/vk \cong \Z^d$ by Fact \ref{fact: r-quasi-divisorial facts}, hence $\I_v^n \cong \Lambda_n^d$.
	Let $\sigma_1,\ldots,\sigma_d$ be the (independent) generators of $\I_v^n$.
	By Lemma \ref{lemma: residue C-pair}(3), we see that $\{\sigma_1,\ldots,\sigma_d\}$ is a C-set.

	Now suppose that $\sigma_1,\ldots,\sigma_r$ are $\Lambda_n$-independent, and that $\{\sigma_1,\ldots,\sigma_r\}$ is a C-set.
	We must show that $r \leq d$.
	Note that $(\sigma_1)_1,\ldots,(\sigma_r)_1$ are $\Lambda_1$-independent in $\gfrak^1(K)$, since $\gfrak^n(K)/\ell = \gfrak^1(K)$ (arguing similarly to Lemma \ref{lemma: m-lifts}(1)).
	Also, note that $\{(\sigma_1)_1,\ldots,(\sigma_r)_1\}$ is a C-set in $\gfrak^1(K)$.
	Therefore, it suffices to assume that $n = 1$, so that $\sigma_i = (\sigma_i)_1$ for $i = 1,\ldots,r$.

	Let $\Sigma$ denote the set of valuative elements of $\langle \sigma_1,\ldots,\sigma_r\rangle_{\Lambda_1} =: \Delta$.
	By Lemma \ref{lemma: C-pair implies comparable}, we see that the valuations $(v_\sigma)_{\sigma \in \Sigma}$ are pair-wise comparable.
	Hence $\Sigma$ itself is valuative by Lemma \ref{lemma: supremum}, and therefore $\Sigma = \langle \Sigma \rangle_{\Lambda_1}$.
	By Theorem \ref{theorem: main theorem of C-pairs}, it follows that $\Delta/\Sigma$ is cyclic.
	Moreover, if $\tau \in \Delta$ and $\sigma \in \Sigma$ then by Lemma \ref{lemma: C-pair implies decomposition} we know that $\tau \in \D_{v_\sigma}^1$.
	Hence $\Delta \subset \D_{v_\Sigma}^1$ by Lemma \ref{lemma: supremum}.

	Put $v := v_\Sigma$.
	If $\Sigma = \Delta$, then $r \leq \rl(vK) \leq \rr(vK/vk) \leq d$ since $\Sigma = \Delta \subset \I_v^1$.
	On the other hand, suppose that $\Sigma \neq \Delta$, and thus $\gfrak^1(Kv) \neq 1$ by Lemma \ref{lemma: residue C-pair}(1).
	This implies that $\trdeg(Kv|kv)\geq 1$ since $kv$ is algebraically closed.
	But on the other hand, we know that $\rl(vK) \geq r-1$ since $\Sigma \subset \I_v^1$ and $\Sigma$ has rank $\geq r-1$.
	Thus, we have 
	\[ r \leq \rl(vK) + \trdeg(Kv|kv) \leq d, \]
	as required.
\end{proof}

Before we prove Theorem \ref{maintheorem: quasi-prime-divisors}, we will prove the following useful lemma.
\begin{lemma}
\label{lemma: d-1-prime-divisors}
	Let $n \in \Nbar$ and $N \geq \Rbf(n)$ be given.
	Let $K$ be a function field over an algebraically closed field $k$ such that $d := \trdeg(K|k) \geq 2$.
	Suppose that $\sigma_1,\ldots,\sigma_{d-1} \in \gfrak^n(K)$ are given such that $\Sigma := \{\sigma_1,\ldots,\sigma_{d-1}\}$ satisfies the equivalent conditions of Theorem \ref{maintheorem: definability of visible inertia groups}, and such that $\langle \Sigma \rangle_{\Lambda_n}$ has rank $d-1$.
	Then there exists an almost-$(d-1)$-quasi-divisorial valuation $v$ of $K|k$ such that $\D^N_n(\Sigma) = \D_v^n$ and $\I^N_n(\D^N_n(\Sigma)) = \I_v^n$.
\end{lemma}
\begin{proof}
	By Theorem \ref{maintheorem: definability of visible decomposition groups}, there exists an $n$-visible valuation $v$ of $K$ such that $\D^N_n(\Sigma) = \D_v^n$, and $\Sigma \subset \I^N_n(\D^N_n(\Sigma)) = \I_v^n$.
	In particular, our assumptions ensure that $\rl(vK) \geq d-1$.

	On the other hand, $\gfrak^n(Kv)$ is non-trivial since $v$ is $n$-visible.
	Since $kv$ is algebraically closed, this implies that $\trdeg(Kv|kv) \geq 1$.
	Thus, we have the following inequality:
	\[ d \leq \rl(vK) + \trdeg(Kv|kv) \leq \trdeg(K|k) = d. \]
	In particular, $v$ has no transcendence defect.
	Finally, one has 
	\[ d-1 \leq \rl(vK) \leq \rr(vK/vk) \leq d-1 \]
	hence $d-1 = \rr(vK/vk)$.
	Since $vK$ contains no non-trivial $\ell$-divisible convex subgroups by the fact that $v$ is $n$-visible, we deduce that $v$ is an almost-$(d-1)$-quasi-divisorial valuation of $K|k$, as required.
\end{proof}

\subsection{Proof of Theorem \ref{maintheorem: quasi-prime-divisors}}
We now turn to the proof of Theorem \ref{maintheorem: quasi-prime-divisors}, and we use the notation introduced in the statement of the theorem.

\vskip 10pt
\noindent$(1) \Rightarrow (2)$.
First of all, by condition (a) and Theorem \ref{maintheorem: definability of visible inertia groups}, we know that $\sigma_1 \in \Ivis^n(K)$.
Thus, by Theorem \ref{maintheorem: definability of visible decomposition groups} and condition (c), there exists an $n$-visible valuation $v$ of $K$ such that $I = \I_v^n$ and $D = \D_v^n$.

Moreover, by condition (a) and Theorem \ref{maintheorem: definability of visible decomposition groups}, there exists an $n$-visible valuation $w$ of $K$ such that $\{\sigma_1,\ldots,\sigma_{d-1}\} \subset \I_w^n$.
Since $\I_v^n = \Lambda_n \cdot \sigma_1$ by (c), it follows from Lemma \ref{lemma: valuative} that $v = v_{\sigma_1}$ and therefore $v$ is a coarsening of $w$.
On the other hand, by Lemma \ref{lemma: d-1-prime-divisors}, we know that $w$ is almost-$(d-1)$-quasi-divisorial using condition (b).

Finally, we note that $v$ is a coarsening of $w$ by Lemma \ref{lemma: valuative}, and thus $v$ also has no transcendence defect by Fact \ref{fact: defectless coarsening}.
Condition (V1) says that $vK$ has no non-trivial $\ell$-divisible convex subgroups.
Thus, it follows from Fact \ref{fact: r-quasi-divisorial facts}(1) that $vK/vk \cong \Z^r$ for $r = \rr(vK/vk)$.
But then $r$ is the rank of $\I_v^n$, which is $1$ by assumption.
Hence, $v$ is a quasi-divisorial valuation, as required.

\vskip 10pt
\noindent$(2) \Rightarrow (1)$.
Suppose that $v$ is a quasi-divisorial valuation of $K$.
Thus $Kv$ is a function field of transcendence degree $d-1$ over $kv$ by Fact \ref{fact: r-quasi-divisorial facts}.
Let $w_0$ be an almost-$(d-2)$-prime-divisor of $Kv|kv$.
Then $w := v \circ w_0$ is an almost-$(d-1)$-quasi-prime-divisor of $K|k$ by Fact \ref{fact: r-quasi-divisorial facts}(3), and $v$ is a coarsening of $w$.
To conclude (1), we take $\sigma_1$ to be a generator of $\I_v^n$, and $\sigma_2,\ldots,\sigma_{d-1} \in \I_w^n$ such that 
\[ \I_w^n = \langle \sigma_1,\ldots,\sigma_{d-1} \rangle_{\Lambda_n}. \]
Condition (c) follows from Theorem \ref{maintheorem: definability of visible decomposition groups}, and (b) clearly holds true since $wK/wk \cong \Z^{d-1}$. 
Condition (a) follows from the fact that $w$ is an $n$-visible valuation (Lemma \ref{lemma: quasi-divisors are visible}).

\bibliographystyle{amsalpha}
\bibliography{../../../refs} 
\end{document}